\documentclass{article}
\usepackage[english]{babel}
\usepackage{tikz}
\usepackage{physics}
\usepackage[letterpaper,top=2cm,bottom=2cm,left=3cm,right=3cm,marginparwidth=1.75cm]{geometry}

\usepackage{amsmath, amsfonts, amsthm}
\usepackage{graphicx}
\usepackage[colorlinks=true, allcolors=blue]{hyperref}
\usepackage{mathtools} 
\allowdisplaybreaks 
\usepackage{comment} 
\usepackage[capitalize]{cleveref}
\usepackage{bbm}
\usepackage{float}
\usepackage{booktabs}
\usepackage{nicematrix}
\usepackage{tikz}
\usepackage{subcaption}
\usepackage{xcolor}
\usepackage{tikz-cd}
\usepackage[ruled]{algorithm2e}
\usepackage[style=alphabetic,maxnames=99,
    maxbibnames=99]{biblatex}
\usepackage{titlesec}
\usepackage{ytableau}
\usepackage{authblk}
\usetikzlibrary{positioning,calc,arrows.meta,shapes.geometric}

\definecolor{tnink}{HTML}{0B2D4D}
\definecolor{tnlightblue}{HTML}{A6D8F5}
\definecolor{tnblue}{HTML}{0072B2}
\definecolor{tngrey}{HTML}{808080}
\definecolor{tnlightgrey}{HTML}{D9D9D9}
\definecolor{tngreen}{HTML}{009E73}
\definecolor{tnorange}{HTML}{E69F00}
\definecolor{tnred}{HTML}{D55E00}
\definecolor{tnpurple}{HTML}{CC79A7}
\definecolor{tncyan}{HTML}{56B4E9}

\tikzset{
  leg/.style={draw=tnink, line width=1.2pt, line cap=round, line join=round},
  tensor/.style={
  rectangle,
  draw=tnink,
  line width=1.1pt,
  minimum size=8mm,
  inner sep=0pt,
  text width=8mm,
  align=center,
  fill=#1,
  text=tnink
  },
  tensor/.default=tnblue,
  op/.style={rectangle, rounded corners=1.2mm, draw=tnink, line width=1.1pt, minimum width=9mm, minimum height=7mm, inner sep=1pt, fill=#1, text=white},
  op/.default=tncyan,
  iso/.style={regular polygon, regular polygon sides=3, draw=tnink, line width=1.1pt, minimum size=10mm, inner sep=0pt, fill=#1, text=white},
  iso/.default=tnorange,
  lab/.style={font=\small, text=black},
}

\newtheorem{theorem}{Theorem}
\newtheorem{lemma}[theorem]{Lemma}
\newtheorem{corollary}[theorem]{Corollary}
\newtheorem{proposition}[theorem]{Proposition}

\theoremstyle{definition}

\newtheorem{remark}[theorem]{Remark}

\newtheorem{example}[theorem]{Example}

\DeclareMathOperator{\Gr}{Gr}
\DeclareMathOperator{\Herm}{Herm}
\DeclareMathOperator{\Wg}{Wg}
\DeclareMathOperator{\U}{U}

\newcommand{\R}{\mathbb{R}}

\newcommand{\N}{\mathbb{N}}

\newcommand{\C}{\mathbb{C}}
\newcommand{\vol}{\mathrm{vol}}

\newcommand{\St}{\mathrm{St}}
\renewcommand{\P}{\mathbb{P}}
\newcommand{\id}{\mathrm{id}}
\newcommand{\NJ}{\mathrm{NJ}}

\titleformat{\subsection}[runin]
  {\normalfont\normalsize\bfseries}
  {\thesubsection}
  {1em}
  {}[.]

\titlespacing*{\subsection}
  {0pt}{\baselineskip}{1em}

\title{\bf Degree of tensor train varieties \\ via integral geometry}

\author{
Andrea Rosana%
\thanks{
Max Planck Institute for Mathematics in the Sciences, Leipzig, Germany.
Email: \texttt{andrea.rosana@mis.mpg.de}.
}
\qquad
Otto T.P. Schmidt%
\thanks{
Max Planck Institute for Mathematics in the Sciences, Leipzig, Germany and
INO-CNR Pitaevskii BEC Center and Dipartimento di Fisica, Universit\'a di Trento, Trento, Italy.
Email: \texttt{otto.schmidt@mis.mpg.de}.
}
}

\begin{document}
\maketitle

\begin{abstract}
In this work we consider tensor train varieties. These are varieties of tensors arising in a range of fields, including quantum many-body physics and machine learning. Using methods from integral geometry, we obtain a combinatorial expression for their degrees. We provide the ready-to-use \texttt{julia} package \texttt{TTVarietyDegree.jl}. 

\end{abstract}

\tableofcontents

\section{Introduction}

Tensor decompositions provide algebraic models for high-dimensional data with hidden
low-dimensional structure. Next to (hierarchical) Tucker decomposition \cite{kolda_2009, cheng_2019}, the tensor train (TT) decomposition is one of
the most widely used formats \cite{Oseledets_2011}. Consider a projective tensor $[T]\in
\P\!\left(
\C^{d_1}\otimes\cdots\otimes \C^{d_N}
\right)$ 
and fix a vector of positive integers
\[
D=(D_0,D_1,\dots,D_{N-1},D_N) \in \N^{N+1},
\qquad
D_0=D_N=1.
\]
We refer to $d=(d_1,\dots,d_N)\in \N^N$ as the \emph{local dimensions}, while the integers $D=(D_1,\dots,D_{N-1})$ are called the tensor train ranks or \emph{bond dimensions}. 
\\
A \emph{tensor train (TT) variety} can be defined via rank restrictions on flattenings. The $i$-th flattening of a tensor $T \in \C^{d_1}\otimes\dots\otimes \C^{d_N}$ is the representation of $T$ as the linear map 
\begin{equation}\label{eq:flattening_intro}
    T^{(i)} : (\C^{d_{i+1}}\otimes\dots\otimes \C^{d_N})^* \longrightarrow (\C^{d_1}\otimes\dots\otimes \C^{d_i}).
\end{equation}
The tensor train variety of \emph{signature} $(D,d)$ is the projective algebraic variety defined by \cite[Lemma 2.2]{Borovik_2025} 
\begin{equation}\label{eq:TTdef}
    V_{D, d} = \{[T]\in
\P\!\left(
\C^{d_1}\otimes\cdots\otimes \C^{d_N}
\right)\,|\,\rank(T^{(r)}) \leq D_r\,,\ r=1,... ,N-1\}\,. 
\end{equation}
The bond dimensions $D$ are called \emph{admissible} for $d$ if
\[
D_r\leq D_{r-1}d_r
\qquad\text{and}\qquad
D_r\leq d_{r+1}D_{r+1},
\qquad r=1,\dots,N-1.
\]
We denote $V^=_{D, d}$ the full-rank stratum of $V_{D,d}$, where each rank condition in \eqref{eq:TTdef} is satisfied with equality.
\\

The dimension of $V_{D,d}$ is computed in \cite{Bernardi_2022}. It is natural to study other invariants, such as the degree. In this paper, we derive an exact procedure for computing the degree of TT varieties. Drawing ideas from \cite{Breiding_2024}, where the authors derive a formula for the degree of tensor subspace varieties \cite{landsberg_2012}, we base our approach on integral geometry. Although the Tucker decomposition considered in \cite{Breiding_2024} is similar, their results are not applicable to tensor trains. Instead, the specific format of tensor trains forces us to introduce a recursive construction of the varieties (see Section \ref{sec:tail_var_and_psi_r}).
\\
\noindent
We will use the following result, which is a consequence of the projective kinematic formula \cite{LerarioNotes, howard_1993}. It links the degree of a projective $m$-dimensional algebraic variety $X \subset \P^c$ to its volume (with respect to the Fubini-Study metric; see Section \ref{sec:preliminaries}):
\begin{equation}\label{eq:kinematic_formula}
    \deg(X)
=
\frac{\vol(X)}{\vol(\P^m)}.
\end{equation}

\noindent 
We denote by ${\rm Gr}(k,n)$ the Grassmannian of $k$-dimensional linear spaces in $\C^n$. Its degree is given in \eqref{eq:gr_degree} (see Section \ref{sec:preliminaries}). Our main result is the following.
\begin{theorem}[Degree of tensor train varieties]\label{thm:degree_formula}
Let $V_{D,d}\subset\P\!\left(\C^{d_1}\otimes\cdots\otimes\C^{d_N}\right)$ be the tensor train variety of signature \((D,d)\) and define $m_r:=D_{r-1}d_r-D_r$.
The degree of the tensor train variety is 
\[
\deg(V_{D,d})
=
\frac{
\displaystyle
\prod_{r=1}^{N-1}
\deg\bigl(\Gr(D_r,D_{r-1}d_r)\bigr)
}{
\displaystyle
\prod_{r=1}^{N-1}
(D_rm_r)!
}
\,
f(P_{D,d}),
\]
where $f(P_{D,d})$ is a purely combinatorial factor depending only the signature $(D,d)$ (see \eqref{eq:super_long_f}).
\end{theorem}
The main result in Theorem \ref{thm:degree_formula} allows an algorithmic implementation. We provide the registered and ready-to-use software package \texttt{TTVarietyDegree.jl} in \texttt{julia}. All details, a user manual and instructions for importing the package are provided in the \texttt{GitHub} repository 
\[\texttt{https://github.com/OTPS3141/TTVarietyDegree.jl}.\]
We benchmarked the results of our algorithm comparing it with available numerical methods in \texttt{HomotopyContinuation.jl} and symbolic ones in \texttt{Macaulay2}. For all the examples using reasonable computational resources, we were able to verify the correctness of our result (see Section \ref{sec:implementation}). 

\begin{example}
To highlight the efficiency of our implementation, we compute the degree for a few signatures. We start with examples in $N=3$ as seen in Table \ref{tab:degree_ex}.

\begin{table}[H]
\centering
\begin{tabular}{cccc}
\toprule
\multicolumn{1}{c}{} & \multicolumn{3}{c}{\textbf{$d$}} \\
\cmidrule(rl){2-4} 
\textbf{$D$} & \((3,3,3)\) & \((5,5,5)\) & \((7,7,7)\)  \\
\midrule
\((1,2,2,1)\) & 306 & 347331250 & 1175573068417440  \\
\((1,3,3,1)\) & 1 & 8380140000 & 594997569835451447952  \\
\bottomrule
\end{tabular}
\caption{Some degrees of $V_{D,d}$ for $N=3$ computed using \texttt{TTVarietyDegree.jl}.}
\label{tab:degree_ex}
\end{table}

\noindent
The computation of degrees of higher order tensor trains is also possible. Consider the order-$12$ tensors of signature $d = (2, 2, 2, 2, 2, 2, 2, 2, 2, 2, 2, 2)$ and $D = (1, 2, 2, 2, 2, 2, 2, 2, 2, 2, 2, 2, 1)$.
Using \texttt{TTVarietyDegree.jl}, the degree evaluates to $\deg(V_{D,d}) = 138254733723634727792624640 \approx 1.4\times 10^{26} $. This computation requires less than $400 \mathrm{MB}$ RAM  and terminates in less than a minute on a personal laptop.
\end{example}

Tensor train decompositions of tensors are also known in quantum many-body physics as \emph{matrix product states (MPS) with open boundary conditions}. MPS are central in the study of one-dimensional quantum many-body
systems, for instance in density matrix renormalization group (DMRG) routines \cite{White1992,Schollwoeck2011,Orus2014}, but also in fields like machine learning \cite{cheng_2019}.\\

Information about the degree of TT varieties is both computationally and
geometrically relevant. In physics, one may view it as an algebraic measure for the complexity
of reconstructing a quantum state from linear observables under a fixed
tensor train signature. Consider an unknown state $\eta \in V_{D, d}$ that we want to characterize. To do so, we perform generic measurements, that is we project the unknown state onto known test states $\phi_i$,
\begin{equation*}
    \langle \phi_i , \eta \rangle = \ell_i.
\end{equation*}
Each projection imposes a linear constraint $\langle \phi_i , \eta \rangle - \ell_i = 0$ on the unknown state $\eta$. After $\dim (V_{D, d})$-many measurements, the number of states satisfying these projection constraints is exactly the degree of $V_{D, d}$. Thus, the degree quantifies the identifiability of a state under linear measurements.  \\

The paper is organized as follows. In Section \ref{sec:preliminaries} we introduce the used language and notation and discuss the main tool we use to compute volumes, the smooth coarea formula. In Section \ref{sec:tail_var_and_psi_r} we define a recursive parametrization by introducing tail varieties and their associated one-step maps. Section \ref{sec:NJ_fiber_volume} establishes key results about metric-related properties of one-step maps. In Section \ref{sec:recursive_volume_comp}, we apply the smooth coarea formula to derive a recursive expression for the volume of TT varieties. In Section \ref{sec:degree} we evaluate the volume and conclude the degree computation via the kinematic formula. Finally, we illustrate the algorithmic implementation of our result and compare it to existing methods (Section \ref{sec:implementation}). \\

\textbf{Acknowledgements} We thank Paul Breiding, Viktoriia Borovik and Simon Telen for useful discussions and feedbacks. OTPS has been supported by European Union’s HORIZON–MSCA-2023-DN-JD programme
under the Horizon Europe (HORIZON) Marie Sklodowska-Curie Actions, grant agreement
101120296 (TENORS).  

\section{Preliminaries}\label{sec:preliminaries}


Throughout this paper, whenever a vector space $V$ is endowed with an inner product $\langle\cdot , \cdot \rangle$, orthogonality is understood with respect to it. If $U \subset V$ is a subspace, its orthogonal complement is denoted by $U^{\perp}=\{v \in V \ | \ \langle v,u\rangle=0 \ \  \forall u \in U\}$. \\

Denote by $\U(c+1)$ the unitary group of size $c+1$. It acts transitively on the (complex) projective space $\P^{c}$. There is a unique, up to scaling, $\U(c+1)$-invariant Riemannian metric for this action (\cite{Cartan_1952}), defined as follows. Fix a representative $x \in \C^{c+1}$ for $[x] \in \P^c$. The tangent space to $\P^c$ at $[x]$ can be identified with the orthogonal $x^{\perp}=\{y \in \C^{c+1} \ | \ \langle x,y \rangle =0\} \cong T_{[x]}\P^c$, where $\langle\cdot ,\cdot \rangle$ denotes the standard Hermitian inner product on $\C^{c+1}$. The \emph{Fubini-Study inner product} on $T_{[x]}\P^c$ is given by 
\begin{equation}\label{eq:FS_product}
    \langle v,w\rangle_{FS} = \frac{{\rm Re}(\langle v,w\rangle)}{\|x\|}, \quad \forall v,w \in x^{\perp}\cong T_{[x]}\P^c \, .
\end{equation}
We refer to the corresponding Riemannian metric as the \emph{Fubini-Study metric} on $\P^c$ and denote it by $g_{FS}$. The induced volume measure $\vol(\cdot)$ inherits the $\U(c+1)$-invariance. We recall that the Fubini-Study volume of the projective space is given by 
\begin{equation}\label{eq:proj_volume}
    \vol(\P^c)=\frac{\pi^c}{c!}\, . 
\end{equation}
Given a smooth manifold $X \subset \P^c$, we endow $X$ with the Riemannian metric induced by the Fubini-Study metric, i.e., for $[x]\in X$ we restrict the inner product \eqref{eq:FS_product} to the tangent space $T_{[x]}X \subset T_{[x]}\P^c$. We denote the induced volume measure on $X$ by $\vol_X(\cdot)$. \\

The \emph{Stiefel manifold} of orthonormal $k$-frames of $\C^n$ is denoted by 
\begin{equation*}
    \St(k,n)=\{M \in \C^{n\times k} \ | \ M^*M = I \} \subset \C^{n\times k}.
\end{equation*}
In other words, the matrices $M \in \St(k,n)$ have orthonormal columns with respect to the standard Hermitian product. The unitary group $\U(n)$ acts on $\St(k,n)$ via left-multiplication, with isotropy subgroup given by $\U(n-k)$. This action realizes $\St(k,n)$ as the smooth homogeneous manifold $\U(n) / \U(n-k)$, hence its real dimension is given by 
\begin{equation}\label{eq:stiefel_dim}
    \dim_{\R}(\St(k,n))=\dim_{\R}(\U(n)) - \dim_{\R}(\U(n-k)) = 2nk - k^2.
\end{equation}
We consider the inner product on $\C^{n\times k}$ defined by taking the real part of the Frobenius product:
\begin{equation}\label{eq:real_frobenius}
    \langle A,B\rangle_F : = {\rm Re}(\rm tr(A^*B)). 
\end{equation}
It induces the $\U(n)$-invariant Riemannian metric $g_{F}$ on $\St(k,n) \subset \C^{n\times k}$.  
Recall that the Grassmannian of $k$-dimensional linear spaces in $\C^n$ is related to the Stiefel manifold by the quotient $\Gr(k,n) \cong \St(k,n)/\U(k)$, since for any $U \in \U(k)$ and $M \in \St(k,n)$ the column span of $M$ and $MU$  are the same. Seen as a projective variety in $\P(\Lambda^k\C^n)$ via the Pl\"ucker embedding, $\Gr(k,n)$ inherits a Riemannian metric. With respect to it, the quotient map $\St(k,n) \rightarrow \St(k,n)/\U(k) \cong  \Gr(k,n)$ is a Riemannian submersion, hence
\begin{equation}\label{eq:vol_gr}
    \vol(\Gr(k,n)) = \frac{\vol(\St(k,n))}{\vol(\U(k))}\,,
\end{equation}
where $\vol(\St(k,n))$ and $\vol(\U(k))$ both refer to the volume measures induced by \eqref{eq:real_frobenius} on $\St(k,n)$ and $\U(k)$ respectively. We use the shorthand notation $dM$ to denote $d\vol_{\St(k,n)}(M)$. The normalized volume form on $\St(k,n)$ is instead denoted by $\mu$ and we refer to $d\mu(M)$ as the \emph{normalized Haar measure}.   \\

Finally, given a product of Riemannian manifolds $(P_1,g_1)\times \dots \times (P_s,g_s)$, we endow it with the product Riemannian metric $g = g_1\times\dots\times g_s$. Denoting $P=P_1\times\dots\times P_s$ and $x=(x_i)_{i=1}^s \in P$, recall that $T_xP = T_{x_1}P_1\times\dots\times T_{x_s}P_s$. Then, given $v=(v_i)_{i=1}^s$ and $w=(w_i)_{i=1}^s$ in  $T_{x}P$, we have 
\begin{equation*}
    g_x(v,w):= (g_1)_{x_1}(v_1,w_1) + \cdots  +(g_s)_{x_s}(v_s,w_s) \,.
\end{equation*}

\subsection{The smooth coarea formula}\label{sec:smooth_coarea}
Given a smooth map $\varphi:X \rightarrow Y$ between smooth manifolds, we denote by $D_x\varphi : T_xX \rightarrow T_{\varphi(x)}Y$ the differential of $\varphi$ at $x\in X$. The map $\varphi$ is a \emph{submersion} if the differential $D_x\varphi$ is surjective for every $x \in X$; in particular $\dim(X)\geq \dim(Y)$ and for every $y \in Y$ the fiber $\varphi^{-1}(y)$ is either a smooth submanifold of $X$ of dimension $\dim(X)-\dim(Y)$ or empty. \\
If $\varphi:X\rightarrow Y$ is a smooth submersion between Riemannian manifolds, the restriction of $D_x\varphi$ to $\mathcal K_x^{\perp} :=(\ker(D_x\varphi))^{\perp}$ is an isomorphism. We define the \emph{normal Jacobian} of $\varphi$ at $x$ as its determinant
\begin{equation*}
    \NJ(\varphi,x) := | \det(D_x\varphi|_{\mathcal K_x^{\perp}})|\,,
\end{equation*}
computed with respect to \emph{orthonormal} bases. Given \emph{any} basis $\{v_1,\dots,v_l\}$ of $\mathcal K_x^{\perp}$, denote by $G$ the Gram matrix of the basis elements and by $R$ the Gram matrix of their images $\{D_x\varphi(v_1),\dots,D_x\varphi(v_l)\}$ under the differential. Then the normal Jacobian is equivalently given by
\begin{equation}\label{eq:NJ_any}
    \NJ(\varphi,x)=\sqrt{\frac{\det(R)}{\det(G)}} \, .
\end{equation}

We now state the smooth coarea formula. Its most general version for Lipschitz maps can be found in \cite{Federer_1969}. For a proof in this simpler setting we refer to \cite[Section A-2]{howard_1993}.

\begin{theorem}\label{thm:smooth_coarea}
    Let $\varphi:X\rightarrow Y$ be a smooth submersion between Riemannian manifolds and let $h:X\rightarrow \R$ be measurable. Then  
    \begin{align*}
        \int_X h(x)\,\NJ(\varphi,x)\, d\vol_X(x) = \int_Y\left(\int_{\varphi^{-1}(y)} h(x)\, d\vol_{\varphi^{-1}(y)}(x)\right)d\vol_Y(y).
    \end{align*}
\end{theorem}

Now let $\varphi:X\rightarrow Y$ be as above and let instead $h:Y \rightarrow\R$ be measurable and such that the composition $\tilde h:=h\circ\varphi:X \rightarrow \R$ is measurable. Applying Theorem \ref{thm:smooth_coarea} to $\tilde h$, one obtains as a corollary
\begin{equation}\label{eq:coarea_general_clean}
    \int_X (h\circ\varphi)(x)\,\NJ(\varphi,x)\,d\vol_X(x) = \int_Y h(y) \vol_X(\varphi^{-1}(y))\,d\vol_Y(y)\, .
\end{equation}
We will use formula \eqref{eq:coarea_general_clean} to compute volumes of smoothly parametrized manifolds.

\subsection{Parametrization of tensor train varieties}\label{sec:pre_parametrization}
We define a parametrization for the tensor train variety as follows. For $r=1,\dots,N-1$ define $n_r:= d_r D_{r-1}$ and let
\[
\mathcal P
:=
\prod_{r=1}^{N-1}\St(D_r,n_r)
\times
\P^=\!\left(\C^{D_{N-1}\times d_N}\right),
\]
where $\P^{=}\!\bigl(\C^{D_{N-1}\times d_N}\bigr)$ is the projectivization of the space of full rank $D_{N-1}\times d_N$ complex matrices. Note that $\mathcal P$ is a smooth Riemannian manifold.
We decompose $M^r \in \St(D_r,n_r) \subset \C^{n_r\times D_r}$ in $d_r$ blocks of size $D_{r-1}\times D_r$ as
\begin{equation}\label{eq:stiefelblocks}
M^r=
\begin{bmatrix}
M^r_1\\
\vdots\\
M^r_{d_r}
\end{bmatrix},
\qquad
M^r_{j_r}\in \C^{D_{r-1}\times D_r}.
\end{equation}
Similarly we think of $C \in \P^=(\C^{D_{N-1}\times d_N})$ as a block matrix
\[
C=
\begin{bmatrix}
C_1 & \cdots & C_{d_N}
\end{bmatrix} \in \P^{=}\!\bigl(\C^{D_{N-1}\times d_N}\bigr),
\qquad
C_j\in \C^{D_{N-1}\times 1}.
\]

\noindent As introduced earlier, we define the full-rank stratum $V^{=}_{D,d}$ of the TT variety $V_{D,d}$ as
\[
V^{=}_{D,d}
:=
\left\{
[T]\in
V_{D,d}
:
\rank(T^{(r)})=D_r
\text{ for } r=1,\dots,N-1
\right\}.
\]
It is a smooth Riemannian manifold with the metric induced by the ambient space $\P(\C^{d_1} \otimes \cdots \otimes \C^{d_N})$. We define the \emph{tensor train map} as the smooth surjective submersion between Riemannian manifolds
\[
\Phi:\mathcal P\longrightarrow
V^{=}_{D,d} \subset\P\!\left(
\C^{d_1}\otimes\cdots\otimes \C^{d_N}
\right), \quad
(M^1,\dots,M^{N-1},[C])\longmapsto [T], 
\]
where the tensor is given as the matrix contraction $T_{j_1,\dots,j_N}
=
M^1_{j_1}M^2_{j_2}\cdots M^{N-1}_{j_{N-1}}C_{j_N}$ (see Figure 
\ref{fig:TT_decomp}). 
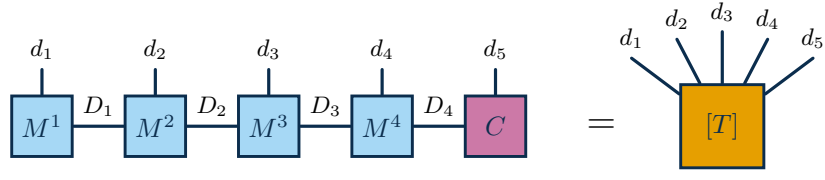
\begin{figure}[]
\centering
\begin{tikzpicture}[x=1cm,y=1cm]

  \begin{scope}
    \foreach \i/\c/\name in {
      0/tnlightblue/M^1,
      1/tnlightblue/M^2,
      2/tnlightblue/M^3,
      3/tnlightblue/M^4,
      4/tnpurple/C
    }{
      \pgfmathtruncatemacro{\site}{\i+1}
      \node[tensor=\c] (T\i) at (1.5*\i,0) {$\name$};
      \draw[leg] (T\i) -- ++(0,0.75)
        node[lab,above] {$d_{\site}$};
    }

    \foreach \i in {0,1,2,3}{
      \pgfmathtruncatemacro{\j}{\i+1}
      \draw[leg] (T\i) -- node[lab,above] {$D_{\j}$} (T\j);
    }
  \end{scope}

  \node[font=\Large] at (7.4,0) {$=$};

  \begin{scope}[xshift=9cm]
    \node[tensor=tnorange, minimum size=11mm] (A) at (0,0) {$[T]$};

    \draw[leg] (A) -- ++(-1.2,0.9)
      node[lab,above] {$d_1$};

    \draw[leg] (A) -- ++(-0.6,1.15)
      node[lab,above] {$d_2$};

    \draw[leg] (A) -- ++(0,1.25)
      node[lab,above] {$d_3$};

    \draw[leg] (A) -- ++(0.6,1.15)
      node[lab,above] {$d_4$};

    \draw[leg] (A) -- ++(1.2,0.9)
      node[lab,above] {$d_5$};
  \end{scope}

\end{tikzpicture}

\caption{A tensor-train representation of a projective tensor $[T]$ for $N=5$.}
\label{fig:TT_decomp}
\end{figure}

Given $[T] = \Phi(M^1,\dots,M^{N-1},[C])$, its fiber is given by 
\begin{equation}\label{eq:fiber_phi}
    \Phi^{-1}([T]) = \{(M^1U_1, U_1^* M^2 U_2, \dots, U_{N-2}^*M^{N-1} U_{N-1}, [U_{N-1}^*C])\, :\, U_i \in \U(D_i)\},
\end{equation}
where the left-action of $U_i^*=U_i^{-1}$ is meant block-wise. This fiber is diffeomorphic to $\prod_{i=1}^{N-1}
\U(D_i)$.
\\ ~ \\
The projective dimension of the TT variety is obtained from a parameter count and subtracting the fiber dimension \cite{Bernardi_2022}. Since $V_{D,d}^=$ is a Zariski-open dense subset of $V_{D,d}$, we have 
\begin{equation}\label{eq:V_dimension}
    \dim_{\C}(V_{D,d})=\dim_{\C}(V_{D,d}^=) = \sum_{i=1}^{N}D_{i}n_i - \sum_{i=1}^{N-1}D_i^2 -1.
\end{equation}

\section{Tail varieties and one-step maps}\label{sec:tail_var_and_psi_r} 

Our goal is to use the kinematic formula \eqref{eq:kinematic_formula} to evaluate the degree of TT varieties. A natural strategy to compute the volume of $V_{D,d}$ is to apply the coarea formula \eqref{eq:coarea_general_clean} to the tensor train map $\Phi:\mathcal P\to V^{=}_{D,d}$. As outlined in Section \ref{sec:smooth_coarea}, this entails the computation of the normal Jacobian of $\Phi$. This, however, requires a description of the orthogonal complement to the kernel of the differential $D_p\Phi$ for every $p \in \mathcal P$. \\ 

\noindent The kernel of $D_p\Phi$ at a point $p \in \mathcal P$ is given by the tangent space to the fiber \eqref{eq:fiber_phi} through it. If we only have one Stiefel factor, that is, $N=2$, the condition to be orthogonal to the kernel is a single Hermiticity constraint; thus, the computation of the normal Jacobian is simple. However, if $N\geq 3$, we get multiple, coupled Hermiticity conditions. This makes an explicit computation unfeasible for larger $N$. \\

The idea is therefore to exploit the simplicity of the $N=2$ case. To do this, we introduce a sequence of smaller tensor train varieties, the \emph{tail varieties}. Each of them is parametrized via \emph{one-step maps}, whose domain is a product between a Stiefel manifold and the previous variety in the sequence. This defines a recursive hierarchy of tail varieties, starting with $\P(\C^{D_{N-1}\times d_N})$ and ultimately yielding the full-rank stratum $V_{D,d}^=$.

\subsection*{Tail varieties}

We now introduce the recursive construction of tensor train varieties via tail varieties and the associated one-step maps. We use a slightly different tensor flattening notation compared to \eqref{eq:flattening_intro}, as there is a shift in the superscript. For a tensor $T_r \in \C^{D_{r-1}} \otimes \C^{d_r}\dots \otimes \C^{d_N}$ and any $s=r-1,\dots,N-1$ let
\[
T_r^{(s)}:\left(\C^{d_{s+1}}\otimes\cdots\otimes \C^{d_N}\right)^*
\longrightarrow 
\C^{D_{r-1}}\otimes
\C^{d_r}\otimes\cdots\otimes
\C^{d_s}
\]
be the flattening after the index $s$, with the convention that for $s=r-1$ the codomain is $\C^{D_{r-1}}$. For any $r=1,\dots,N$ we define the \emph{$r$-th tail variety} 
\[
X_r
:=
\left\{
[T_r]\in \P(\C^{D_{r-1}}\otimes
\C^{d_r}\otimes\cdots\otimes
\C^{d_N}):
\rank\bigl(T_r^{(s)}\bigr)=D_s
\text{ for } s=r-1,\dots,N-1
\right\}.
\]
In particular, $X_1=V^{=}_{D,d}$ is the full-rank stratum of $V_{D,d}$ and $X_N=\P^{=}\!\bigl(\C^{D_{N-1}\times d_N}\bigr)$ (see Fig. \ref{fig:tail_varieties}).
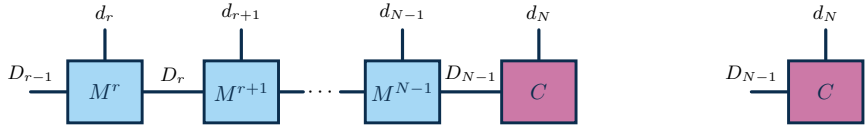
\begin{figure}[H]
\centering

\begin{subfigure}[t]{0.58\textwidth}
\centering
\begin{tikzpicture}[
  x=1cm,y=1cm,
  scale=0.82,
  transform shape,
  tensor/.append style={minimum size=10mm,text width=12mm}
]

\def\s{1.2}
\def\l{1}

\node[tensor=tnlightblue] (Mr)  at (0*\l,0) {$M^r$};
\node[tensor=tnlightblue] (Mrp) at (2.2*\l,0) {$M^{r+1}$};
\node[tensor=tnlightblue] (MNm) at (4.8*\l,0) {$M^{N-1}$};
\node[tensor=tnpurple]    (C)   at (7*\l,0) {$C$};

\draw[leg] (Mr) -- ++(-\s,0) node[lab,above] {$D_{r-1}$};

\draw[leg] (Mr)  -- ++(0,1) node[lab,above] {$d_r$};
\draw[leg] (Mrp) -- ++(0,1) node[lab,above] {$d_{r+1}$};
\draw[leg] (MNm) -- ++(0,1) node[lab,above] {$d_{N-1}$};
\draw[leg] (C)   -- ++(0,1) node[lab,above] {$d_N$};

\draw[leg] (Mr) -- node[lab,above] {$D_r$} (Mrp);

\draw[leg] (Mrp) -- ++(\l,0);
\node at (3.5*\l,0) {$\cdots$};
\draw[leg] (MNm) -- ++(-\l,0);

\draw[leg] (MNm) -- node[lab,above] {$D_{N-1}$} (C);

\end{tikzpicture}

\label{fig:tail_variety_Xr}
\end{subfigure}
\hspace{0em}
\begin{subfigure}[t]{0.25\textwidth}
\centering
\begin{tikzpicture}[
  x=1cm,y=1cm,
  scale=0.82,
  transform shape,
  tensor/.append style={minimum size=10mm,text width=12mm}
]

\def\s{1.2}

\node[tensor=tnpurple] (C) at (0,0) {$C$};

\draw[leg] (C) -- ++(0,1) node[lab,above] {$d_N$};
\draw[leg] (C) -- ++(-\s,0) node[lab,above] {$D_{N-1}$};

\end{tikzpicture}
\label{fig:tail_variety_XN}
\end{subfigure}

\caption{Tensor-network representations of the tail varieties \(X_r\) (left) and \(X_N\) (right).}
\label{fig:tail_varieties}

\end{figure}
\noindent
Note that each tail variety $X_r$ can be seen as a TT variety with adjusted signature $d=(D_{r-1},d_r,...,d_N)$ and $D=(D_{r-1},\dots, D_{N-1}, D_N)$. This also allows us to derive their complex dimension:
\begin{equation}\label{eq:X_r_dimension}
    \dim_{\C}(X_r) := \sum_{i=r}^{N}D_{i}n_i - \sum_{i=r}^{N-1}D_i^2 -1.
\end{equation}

\subsection*{The one-step maps}\label{sec:one_step_map}

We can parametrize the tail variety $X_r$ recursively in terms of the next one $X_{r+1}$ as follows. For \(r=1,\dots,N-1\), let $Y_r := \St(D_r,n_r)\times X_{r+1}$. The \emph{$r$-th one-step map} $\Psi_r:Y_r\longrightarrow X_r$ is defined via
\begin{equation}\label{eq:psi}
\Psi_r(q_r)=[T_r], \ \text{where } (T_r)_{\alpha_{r-1},j_r,\dots,j_N}
=
\sum_{\alpha_r=1}^{D_r}
(M^r_{j_r})_{\alpha_{r-1},\alpha_r}
(T_{r+1})_{\alpha_r,j_{r+1},\dots,j_N} \ ,
\end{equation}
with $q_r := (M^r,[T_{r+1}])\in Y_r$ and $M^r$ with the block structure in \eqref{eq:stiefelblocks}. The aforementioned sequence of tail varieties yielding $V_{D,d}^=$ can therefore be visualized via the sequence of one-step maps
\begin{equation*}
    X_N \overset{\Psi_{N-1}}{\longrightarrow} X_{N-1} \overset{\Psi_{N-2}}{\longrightarrow} \cdots \overset{\Psi_{2}}{\longrightarrow} X_2 \overset{\Psi_{1}}{\longrightarrow} X_1 = V_{D,d}^=\,.
\end{equation*}

The map takes the $(r+1)$-th projective tail variety $X_{r+1}$ and contracts it with the Stiefel manifold $\St(D_r,n_r)$ from the left. This is illustrated in Fig. \ref{fig:one_step_map}. In the following, we use the notation $M^r \cdot T_{r+1}$ to denote the tensor $T_r$ defined by \eqref{eq:psi}. 
The differential of $\Psi_r$ reads
\begin{equation*}
    D_{q_r}\Psi_r(\dot M^r,[\dot T_{r+1}]) = \dot M^r \cdot  T_{r+1} + M^r \cdot \dot T_{r+1},
\end{equation*}
for $(\dot M^r,[\dot T_{r+1}]) \in T_{q_r}Y_r$.
\begin{proposition}\label{prop:Psi_r_submersion}
For every \(r=1,\dots,N-1\), the one-step map $\Psi_r$ is a smooth surjective submersion.
\end{proposition}

\begin{proof}
Smoothness follows from the polynomiality of the map. Surjectivity follows from the left-canonical tensor train decomposition \cite{Schollwoeck2011}:
given \([T_r]\in X_r\), the flattening $T_r^{(r)}$ has rank \(D_r\). Using a singular value decomposition we may factor it as $T_r^{(r)}=M^rB$ for some $M^r \in {\rm St}(D_r,n_r)$ and some $D_r \times (d_{r+1}\dots d_N)$ complex matrix $B$. Since $[T_r] \in X_r$, $B$ is the $r$-th flattening $T_{r+1}^{(r)}$ of some $[T_{r+1}]\in X_{r+1}$. It is straightforward to check that $\Psi_r(M^r,[T_{r+1}])=[T_r]$, thus $\Psi_r$ is surjective. \\
To prove surjectivity of the differential, consider the fiber through a point $(M^r,[T_{r+1}])\in Y_r$. This is the smooth free $\U(D_r)$-orbit
\begin{equation}\label{eq:fiber}
\Psi_r^{-1}(\Psi_r(M^r,[T_{r+1}])) =\left\{(M^rU,[U^{-1}T_{r+1}]): U\in \U(D_r)\right\} \, .
\end{equation}
It is diffeomorphic to $\U(D_r)$. The kernel of the differential of $\Psi_r$ at $(M^r,[T_{r+1}])$ is the tangent space to the fiber \eqref{eq:fiber}. It has real dimension $D_r^2$. The complex Stiefel manifold $\St(D_r, n_r)$ has real dimension $2n_rD_r - D_r^2$ (see \eqref{eq:stiefel_dim}). The claim thus follows from \eqref{eq:X_r_dimension}. 
\end{proof}
\begin{figure}[H]
\centering
\begin{subfigure}[t]{\textwidth}
\centering

\begin{tikzpicture}[
  x=1cm,
  y=1cm,
  scale=0.82,
  transform shape,
  tensor/.append style={
    minimum size=10mm,
    text width=12mm
  }
]


\def\s{1.2}
\def\l{1}

\node[tensor=tnlightblue] (Mr) at (0,0) {$M^r$};

\draw[leg] (Mr) -- ++(0,1)
node[lab,above] {$d_r$};

\draw[leg] (Mr) -- ++(-\s,0)
node[lab,above] {$D_{r-1}$};

\draw[leg] (Mr) -- ++(+\s,0)
node[lab,above] {$D_{r}$};


\node[font=\Large] at (2.1,0) {$\times$};


\begin{scope}[xshift=4.3cm]

\node[tensor=tnlightblue] (Mrp) at (0*\l,0) {$M^{r+1}$};
\node[tensor=tnlightblue] (MNm) at (2.6*\l,0) {$M^{N-1}$};
\node[tensor=tnpurple]    (C)   at (4.8*\l,0) {$C$};


\draw[leg] (Mrp) -- ++(-\s,0)
node[lab,above] {$D_{r}$};


\draw[leg] (Mrp) -- ++(0,1)
node[lab,above] {$d_{r+1}$};

\draw[leg] (MNm) -- ++(0,1)
node[lab,above] {$d_{N-1}$};

\draw[leg] (C) -- ++(0,1)
node[lab,above] {$d_N$};


\draw[leg] (Mrp) -- ++(\l,0);

\node at (1.3*\l,0) {$\cdots$};

\draw[leg] (MNm) -- ++(-\l,0);

\draw[leg] (MNm)
-- node[lab,above] {$D_{N-1}$}
(C);

\end{scope}

\end{tikzpicture}

\label{fig:tail_variety_factorization}

\end{subfigure}
\vfill

\begin{subfigure}[t]{0.58\textwidth}
\centering
\begin{tikzpicture}[
  x=1cm,y=1cm,
  scale=0.82,
  transform shape,
  tensor/.append style={minimum size=10mm,text width=12mm}
]

\def\s{1.2}
\def\l{1}

\node[font=\Large] at (0,0) {$\longrightarrow$};

\node[tensor=tnlightblue] (Mr)  at (2.5*\l,0) {$M^r$};
\node[tensor=tnlightblue] (Mrp) at (4.7*\l,0) {$M^{r+1}$};
\node[tensor=tnlightblue] (MNm) at (7.3*\l,0) {$M^{N-1}$};
\node[tensor=tnpurple]    (C)   at (9.5*\l,0) {$C$};

\draw[leg] (Mr) -- ++(-\s,0) node[lab,above] {$D_{r-1}$};

\draw[leg] (Mr)  -- ++(0,1) node[lab,above] {$d_r$};
\draw[leg] (Mrp) -- ++(0,1) node[lab,above] {$d_{r+1}$};
\draw[leg] (MNm) -- ++(0,1) node[lab,above] {$d_{N-1}$};
\draw[leg] (C)   -- ++(0,1) node[lab,above] {$d_N$};

\draw[leg] (Mr) -- node[lab,above] {$D_r$} (Mrp);

\draw[leg] (Mrp) -- ++(\l,0);
\node at (6*\l,0) {$\cdots$};
\draw[leg] (MNm) -- ++(-\l,0);

\draw[leg] (MNm) -- node[lab,above] {$D_{N-1}$} (C);

\end{tikzpicture}

\label{fig:tail_variety_Xr}
\end{subfigure}

\caption{Tensor-network representation of the one-step map $\Psi_r:
\St(D_r,n_r)\times X_{r+1}
\longrightarrow X_r$.}
\label{fig:one_step_map}

\end{figure}
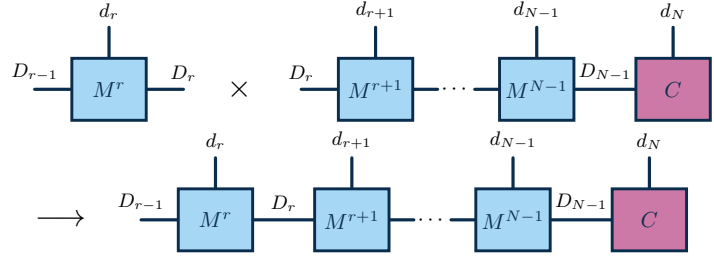

\section{Metric properties of one-step maps}\label{sec:NJ_fiber_volume}

As explained before, we want to apply the smooth coarea formula \eqref{eq:coarea_general_clean} to the one-step maps $\Psi_r$. We therefore need to establish explicit forms for the volume of their fibers and for their normal Jacobian.

\subsection{Volume of the fibers and pullback metric}
We want to derive a formula for the volume of the fibers of $\Psi_r$. The fiber through the point $q_r=(M^r,[T_{r+1}]) \in Y_r$ is parametrized via the smooth free $\U(D_r)$-action
\begin{align*}
\alpha_{q_r}:\U(D_r)&\longrightarrow  \Psi_r^{-1}(\Psi_r(q_r)) \, ,\\
U\quad &\longmapsto (M^rU,[U^{-1}T_{r+1}])\, .
\end{align*}
Its infinitesimal action on the tail variety factor $X_{r+1}$, i.e., the left-action on the tensor $[T_{r+1}]$, is encoded by the linear map $B_r:\mathfrak u(D_r)\longrightarrow W_r $ given by 
\begin{align}\label{eq:B_r_definition}
    B_r(X) :=  \frac{d}{dt}\bigg|_{t=0}
[e^{-tX}T_{r+1}]\,,
\end{align}
with $W_r := T_{[T_{r+1}]}X_{r+1}$ and $X\in T_{I}\U(D_r) = \mathfrak{u}(D_r)$. Let $B_r^*:W_r\to \mathfrak u(D_r)$ denote the adjoint map of \(B_r\). Given $U \in \U(D_r)$, if $D_U\alpha_{q_r}(\eta)=0$ for some $\eta \in T_U\U(D_r)$, this implies $M^r\eta=0$ for the right-action on the Stiefel argument. Since $M^r \in \St(D_r,n_r)$, we find $(M^r)^*M^r\eta = \eta =0$ and we conclude that $\alpha_{q_r}$ is a smooth immersion.
\\~\\
Since $\alpha_{q_r}$ parametrizes the fiber of $\Psi_r(q_r)$, we have that $D_I\alpha_{q_r}(\mathfrak u(D_r))
\subseteq
\ker(D_{q_r}\Psi_r)$. We denote $\mathcal K_r := \ker(D_{q_r}\Psi_r)$. 
Furthermore, $\dim_{\mathbb R}D_I\alpha_{q_r}(\mathfrak u(D_r))
=
\dim_{\mathbb R}\mathfrak u(D_r)
=
D_r^2$, since $D_I\alpha_{q_r}$ is injective. Since $\dim_{\mathbb R}\mathcal K_r=D_r^2$ (see \eqref{eq:fiber}) it follows that $D_I\alpha_{q_r}(\mathfrak u(D_r)) 
=
\mathcal K_r$. Consequently, for any orthonormal real basis
$\{X_a\}_{a=1}^{D_r^2}$ of $\mathfrak u(D_r)$, the vectors $D_I\alpha_{q_r}(X_a)
=
(M^rX_a,B_r(X_a))$
form a basis of $\mathcal K_r$. \\

 On the fiber the metric is induced by the product metric $g_r = g_F \times g_{FS}$ (see Section \ref{sec:preliminaries}). From the description of the fiber we have that $\alpha_{q_r}$ is a diffeomorphism, but it is not an isometry. Hence, in order to compute the volume of $\psi_r^{-1}(\psi_r(q_r))$, we need to understand how the product metric pulls back to $\U(D_r)$, i.e., we need to compute $\alpha_{q_r}^*g_r$.
\begin{proposition}\label{prop:fiber_volume_Psi_r}
The volume of the fiber of \(\Psi_r\) through $q_r=(M^r,[T_{r+1}]) \in Y_r$ is
\[
\vol\bigl(\Psi_r^{-1}(\Psi_r(q_r))\bigr)
=
\vol(\U(D_r))\,\sqrt{\det(I+B_r^*B_r)}. 
\]
\end{proposition}

\begin{proof}
Let $\{X_a\}_{a=1}^{D_r^2}$ be an orthonormal basis of \(\mathfrak u(D_r)\). A basis for the tangent space to the fiber is
\[
K_a:=D_I\alpha_{q_r}(X_a)
=
\bigl(M^rX_a,\;B_r(X_a)\bigr) ,
\qquad a=1,\dots,D_r^2.
\]
 The pullback metric through the smooth immersion $\alpha_{q_r}$ is given by
\begin{equation*}
    \alpha_{q_r}^*g_r(X_a, X_b) = g_r(D_I\alpha_{q_r}(X_a), D_I\alpha_{q_r}(X_b)) =: H^{(r)}_{ab}, \quad X_a, X_b \in \mathfrak u(D_r).
\end{equation*}
Using the product metric and the fact that \(M^r\) has orthonormal columns, we get 
\begin{align*}
H^{(r)}_{ab}
&=
\langle M^rX_a,M^rX_b\rangle_{F}
+
\langle B_r(X_a),B_r(X_b)\rangle_{FS} \\
&=
\langle X_a,X_b\rangle_{F}
+
\langle B_r(X_a),B_r(X_b)\rangle_{FS} \\
&=
\delta_{ab}
+
\langle B_r^*B_r(X_a),X_b\rangle_{F}.
\end{align*}
Thus, in the orthonormal basis $\{X_a\}_{a=1}^{D_r^2}$, the Gram matrix \(H^{(r)}\) is the matrix of
the operator
\[
I+B_r^*B_r:\mathfrak u(D_r)\to \mathfrak u(D_r).
\]
The induced Riemannian volume density on the fiber
relative to the normalized Haar density on \(\U(D_r)\) is therefore
\[
\sqrt{\det(I+B_r^*B_r)}.
\]
Since this ratio is constant on \(\U(D_r)\) (both the reference metric on $\mathfrak u(D_r)$ and the induced metric are right-invariant), the total fiber volume
is
\[
\vol\bigl(\Psi_r^{-1}(\Psi_r(q_r))\bigr)
=
\vol\bigl(\U(D_r)\bigr)\,\sqrt{\det(I+B_r^*B_r)}.
\]
\end{proof}
\noindent
Proposition \ref{prop:fiber_volume_Psi_r} implies that the following function $J_r : X_r \rightarrow \R$ is well-defined: 
\begin{equation}\label{eq:J_r_definition}
J_r([T_{r}])
:=
\sqrt{\det(I+B_r^*B_r)}
\end{equation}
for $[T_r] = \Psi_r(M^r, [T_{r+1}])$. We will make use of this in the next section.

\subsection{Normal Jacobian of one-step maps}
To compute the normal Jacobian of $\Psi_r$ at a point $q_r \in Y_r$, we first need to determine a basis of the orthogonal complement of the kernel of $D_{q_r}\Psi_r$, denoted by $\mathcal K_r^{\perp}$. \\

Let
\(\{z_i\}_{i=1}^{s}\) be an orthonormal basis of \(W_r 
=T_{[T_{r+1}]}X_{r+1}\) with $s:=\dim(X_{r+1})$ and define $h_i:=\bigl(-M^rB_r^*z_i,z_i\bigr)$ for $i=1, ..., s$.
We claim that each \(h_i\) is orthogonal to the kernel elements \((M^rX,B_r(X)) \in \mathcal K_r\) for any $X\in\mathfrak u(D_r)$. Indeed, using the adjoint relation $\langle X,B_r^*z_i\rangle_{F}
=
\langle B_r(X),z_i\rangle_{FS}$, we have
\begin{align*}
\big\langle h_i,(M^rX,B_r(X))\big\rangle
&=
\big\langle -M^rB_r^*z_i,M^rX\big\rangle_{F}
+
\big\langle z_i,B_r(X)\big\rangle_{FS} \\
&=
-\big\langle B_r^*z_i,X\big\rangle_{F}
+
\big\langle z_i,B_r(X)\big\rangle_{FS}= 0.
\end{align*}
Since $\{h_i\}_{i=1}^{s}$ is not an orthonormal set, we introduce its Gram matrix $G=(G_{ij})_{i,j=1}^s$\,, where
\begin{align*}
G_{ij} :=\langle h_i,h_j\rangle
&=
\big\langle -M^rB_r^*z_i,-M^rB_r^*z_j\big\rangle_{F}
+
\langle z_i,z_j\rangle_{FS} \\
&=
\langle B_r^*z_i,B_r^*z_j\rangle_{F}
+
\delta_{ij} \\
&=
\langle B_rB_r^*z_i,z_j\rangle_{FS}
+
\delta_{ij}.
\end{align*}
Thus, in the orthonormal basis \(\{z_i\}_{i=1}^{s}\), the Gram matrix \(G\) is the matrix of the operator
\[
I+B_rB_r^*:W_r\to W_r\,.
\]
To compute the normal Jacobian restricted to ${\rm span}\{h_1,\dots,h_s\}$ we need to evaluate the Gram matrix of the differential images $\{D_{q_r}\Psi_r(h_i)\}_{i=1}^s$. We denote it by $R=(R_{ij})_{i,j=1}^s$. By linearity of the differential,
\begin{equation}\label{eq:diff_lin}
    D_{q_r}\Psi_r(h_i)
=
D_{q_r}\Psi_r(-M^rB_r^*z_i,0)
+
D_{q_r}\Psi_r(0,z_i).
\end{equation}
To simplify the computation of $R$, we prove the following Lemma.
\begin{lemma}\label{lem:xi_M_isometric}
For fixed \(M^r\in \St(D_r,n_r)\), let $\xi_{M^r}:X_{r+1}\to X_r$ be defined by
\[
[T_{r+1}]\longmapsto [M^r \cdot T_{r+1}].
\]
Then \(\xi_{M^r}\) is an isometric immersion with respect to the induced
metrics on \(X_{r+1}\) and \(X_r\).
\end{lemma}

\begin{proof}
Let \([T_{r+1}]\in X_{r+1}\) and choose the normalized representative \(T_{r+1}\) with \(\|T_{r+1}\|_{FS}=1\).
Let \(Z,W\in T_{[T_{r+1}]}X_{r+1}\). Then
\[
\langle D_{[T_{r+1}]}\xi_{M^r}(Z), D_{[T_{r+1}]}\xi_{M^r}(W)\rangle_{FS} = \langle M^r\cdot Z, M^r \cdot W\rangle_{FS} = \langle Z,W\rangle_{FS}, 
\]
hence the conclusion follows.
\end{proof}
\noindent
We now use Lemma \ref{lem:xi_M_isometric} to establish the form of the normal Jacobian restricted to $L:={\rm span}\{h_1,\dots,h_s\}$.
\begin{proposition}\label{prop:NJ_rest_L}
    The normal Jacobian of the differential restricted to $L:={\rm span}\{h_1,\dots,h_s\}$ is
\[
\NJ\left(\Psi_r, q_r\right)|_L
=
\sqrt{\frac{\det R}{\det G}} = \sqrt{\det(I+B_rB_r^*)} = J_r(\Psi_r(M^r,[T_{r+1}])\,.
\]
\end{proposition}
\begin{proof}
    The second term of the differential in \eqref{eq:diff_lin} can be written as
\[
D_{q_r}\Psi_r(0,z_i)
=
M^r\cdot z_i = D_{[T_{r+1}]}\xi_{M^r}(z_i)\,.
\]
Consider the curve $\gamma(t):=\bigl(M^re^{-tB_r^*z_i},[T_{r+1}]\bigr)$ with initial velocity  $\dot\gamma(0)=(-M^rB_r^*z_i,0)$.
Then the first term in \eqref{eq:diff_lin} is given by
\begin{align*}
    D_{q_r}\Psi_r(-M^rB_r^*z_i,0)
&=
\frac{d}{dt}\bigg|_{t=0}
\Psi_r(M^re^{-tB_r^*z_i},[T_{r+1}])  \\
&=
\frac{d}{dt}\bigg|_{t=0}
[M^re^{-tB_r^*z_i}\cdot T_{r+1}]  \\
&=
\frac{d}{dt}\bigg|_{t=0}
\xi_{M^r}([e^{-tB_r^*z_i}T_{r+1}])  \\
&=
D_{[T_{r+1}]}\xi_{M^r}
\left(
\frac{d}{dt}\bigg|_{t=0}
[e^{-tB_r^*z_i}T_{r+1}]
\right)  \\
&=
D_{[T_{r+1}]}\xi_{M^r}\bigl(B_r(B_r^*z_i)\bigr).
\end{align*}
Hence, we conclude that $D_{q_r}\Psi_r(h_i)
=
D_{[T_{r+1}]}\xi_{M^r}\bigl((I+B_rB_r^*)z_i\bigr)$.
By Lemma \ref{lem:xi_M_isometric}, the map \(\xi_{M^r}\) is an isometric immersion, therefore the Gram matrix $R$ of the images is 
\[
R_{ij}:=
\big\langle D_{q_r}\Psi_r(h_i),D_{q_r}\Psi_r(h_j)\big\rangle_{FS} = \big\langle (I+B_rB_r^*)z_i,\,(I+B_rB_r^*)z_j\big\rangle_{F}.
\]
Thus, in the orthonormal basis \(\{z_i\}_{i=1}^s\), the Gram matrix \(R\) is the matrix $(I+B_rB_r^*)^2$. Using \eqref{eq:NJ_any}, the normal Jacobian of the differential restricted to $L:=\mathrm{span}_{\mathbb R}\{h_i\}_{i=1}^{s}$ is therefore given by
\[
\NJ\left(\Psi_r, q_r\right)|_L
=
\sqrt{\frac{\det R}{\det G}} = \sqrt{\det(I+B_rB_r^*)}\,.
\]
We use the standard identity $\det(I+B_rB_r^*)=\det(I+B_r^*B_r)$,
which holds because the operators \(B_rB_r^*\) and \(B_r^*B_r\) have the same nonzero eigenvalues.
Therefore we proved that
\[
\NJ\left(\Psi_r, q_r\right)|_L
=
J_r(\Psi_r(q_r))
= J_r(\Psi_r(M^r, [T_{r+1}])),
\]
where $J_r$ is the same factor \eqref{eq:J_r_definition} appearing in the fiber volume computation. 
\end{proof}
\noindent
However, the linearly independent vectors $\{h_i\}_{i=1}^{s}$ do not span the entire orthogonal complement to the kernel of $D_{{q_r}}\Psi_r$. We can extend $\{h_i\}_{i=1}^{s}$ to a basis as follows. Let $m_r := D_{r-1}d_r - D_r$ and $M^r_{\perp}\in \St(m_r,n_r)$ be an orthonormal frame orthogonal to $M^r$, that is, $(M^r_{\perp})^*M^r=0$.
\begin{proposition}\label{prop:basis_complement}
The orthogonal complement of the kernel of $D_{q_r}\Psi_r$ splits orthogonally as
\begin{equation}\label{eq:K_decomposition}
   \mathcal K_r^{\perp} := (\ker(D_q\Psi_r))^{\perp}
=
\underbrace{
\mathrm{span}_\R\{(M_\perp^rE_{ab},0),(iM_\perp^rE_{ab},0)\}_{a,b=1}^{m_r,D_r}
}_{=:F}
\ \oplus\
\underbrace{\mathrm{span}_{\mathbb R}\{h_i\}_{i=1}^{s}}_{=:L},
\end{equation}
with $E_{ab}$ the $m_r \times D_r$ matrix with a single entry $1$ at index $(a,b)$. 
\end{proposition}
\begin{proof}
Since \(\{h_i\}_{i=1}^s\) are orthogonal to $\ker(D_q\Psi_r)$, we need to show that~$\{(M_\perp^rE_{ab},0),(iM_\perp^rE_{ab},0)\}_{a,b=1}^{m_r,D_r}$ are orthogonal to \(\{h_i\}_{i=1}^s\) and $\ker(D_q\Psi_r)$, and that the blocks $F,\,L$ together define $\dim_{\R}((\ker(D_q\Psi_r))^{\perp})$-many independent basis elements.
\\
First, the orthogonality $\{(M_\perp^rE_{ab},0),(iM_\perp^rE_{ab},0)\}_{a,b=1}^{m_r,D_r} \perp \{h_i\}_{i=1}^s$ follows from $(M^r_{\perp})^*M^r =0$. Similarly, orthogonality to the kernel $\mathcal K_r$ follows from
\begin{equation*}
    \langle (M_\perp^rE_{ab},0), (M^rX, B_r(X))\rangle = \langle M_\perp^rE_{ab}, M^rX\rangle_{F} + \langle 0, B_r(X) \rangle_{FS}=0,
\end{equation*}
for every $X \in \mathfrak u(D_r)$, $a=1,\dots,m_r$ and $b=1,\dots,D_r$.
For a TT signature $(D,d)$ we have $\dim_{\R}(\{(M_\perp^rE_{ab},0),(iM_\perp^rE_{ab},0)\})= 2m_rD_r$ and $\dim_{\R}(T_{M^r}\St(D_r,n_r)) = D_r^2 + 2D_rm_r$ (see \eqref{eq:stiefel_dim}). Then
\begin{align*}
    \dim_{\R}(\ker(D_{q_r}\Psi_r)^{\perp})&= \dim_{\R}(T_{M^r}\St(D_r,n_r)) + \dim_{\R}(T_{[T_{r+1}]}X_{r+1}) - \dim_{\R}(\ker(D_{q_r}\Psi_r)) \\
    &= (D_r^2 + 2D_rm_r) +s -D_r^2  \\
    &= s + 2D_rm_r,
\end{align*}
hence the proposed basis has the expected dimension. 
\end{proof}
\begin{corollary}\label{cor:NJ_Psi_r}
The normal Jacobian of the one-step map \(\Psi_r: Y_r \longrightarrow X_r\) is
\[
\NJ(\Psi_r,(M^r, [T_{r+1}]))
=
\det\bigl(\Sigma_{r+1}([T_{r+1}])\bigr)^{m_r}
\,J_r(\Psi_r(M^r, [T_{r+1}])),
\]
where we denote
\begin{equation}\label{eq:Sigma_r_definition}
\Sigma_r([T_r])
:=
T_r^{(r-1)}
(T_r^{(r-1)})^* \in \Herm_{>0}(D_{r-1}).
\end{equation}
\end{corollary}

\begin{proof}
By Proposition \ref{prop:basis_complement}, we need compute the Gram matrix of the basis of $F$ in \eqref{eq:K_decomposition} and that of its image through the differential $D_{q_r}\Psi_r$. It is straightforward to see that the basis is orthonormal. For the images through the differential we have
\begin{align*}
    \langle D_{q_r}\Psi_r(M^r_{\perp}E_{ab}, 0), D_{q_r}\Psi_r(M^r_{\perp}E_{cd}, 0) \rangle 
&= \langle (M^r_{\perp}E_{ab})\cdot T_{r+1}, (M^r_{\perp}E_{cd})\cdot T_{r+1} \rangle \\
&=\delta_{ac}\Re((\Sigma_{r+1}([T_{r+1}]))_{bd}), \\[4pt]
\langle D_{q_r}\Psi_r(M^r_{\perp}E_{ab}, 0), D_{q_r}\Psi_r(iM^r_{\perp}E_{cd}, 0) \rangle 
&= -\delta_{ac}\Im((\Sigma_{r+1}([T_{r+1}]))_{bd}), \\[4pt]
\langle D_{q_r}\Psi_r(iM^r_{\perp}E_{ab}, 0), D_{q_r}\Psi_r(M^r_{\perp}E_{cd}, 0) \rangle 
&= \delta_{ac}\Im((\Sigma_{r+1}([T_{r+1}]))_{bd}), \\[4pt]
\langle D_{q_r}\Psi_r(iM^r_{\perp}E_{ab}, 0), D_{q_r}\Psi_r(iM^r_{\perp}E_{cd}, 0) \rangle 
&= \delta_{ac}\Re((\Sigma_{r+1}([T_{r+1}]))_{bd}).
\end{align*}
Thus the Gram matrix is given by 
\begin{equation*}A := (D_{q_r}\Psi_r|_F)^*D_{q_r}\Psi_r|_F =
\begin{bmatrix}
        \id_{m_r} \otimes \Re(\Sigma_{r+1}([T_{r+1}])) & -\id_{m_r} \otimes \text{Im}(\Sigma_{r+1}([T_{r+1}])) \\
        \id_{m_r} \otimes \text{Im}(\Sigma_{r+1}([T_{r+1}])) & \id_{m_r} \otimes \Re(\Sigma_{r+1}([T_{r+1}]))
    \end{bmatrix} ,
\end{equation*}
and it follows that the normal Jacobian of the differential restricted to $F$ is 
\begin{align*}
    \NJ\left(\Psi_r, (M^r, [T_{r+1}])\right)|_F &= \sqrt{\det(A)} =\det(\Sigma_{r+1}([T_{r+1}]))^{m_r}.
\end{align*}
Since the blocks $F$ and $L$ are orthogonal, and since also their images through $D_{q_r}\Psi_r$ are orthogonal
\begin{equation*}
    \langle D_{q_r}\Psi_r(M^r_{\perp}E_{ab}, 0), D_{q_r}\Psi_r(M^rX_i, B_r(X_i)) \rangle 
= 0\, ,
\end{equation*}
the normal Jacobian of $\Psi_r$ splits as the product
\begin{align*}
    \NJ(\Psi_r,(M^r, [T_{r+1}]))
&= \NJ\left(\Psi_r, (M^r, [T_{r+1}])\right)|_F \ \NJ\left(\Psi_r, (M^r, [T_{r+1}])\right)|_L \\
&= \det\bigl(\Sigma_{r+1}([T_{r+1}])\bigr)^{m_r}
\,J_r(\Psi_r(M^r, [T_{r+1}])). \qedhere
\end{align*}
\end{proof}

\section{Recursive volume computation}\label{sec:recursive_volume_comp}
In this section we introduce the main recursion relation used in the computation of the volume via \eqref{eq:coarea_general_clean}. The idea is to recursively shift integrations on tail varieties from the "complicated" full TT variety $X_1$ to the "simple" projective space $X_N$. Shifting the integration domain in this fashion will naturally affect the integrand as well. We first show how a recursion of integration domains emerges, and then follow with a relation that illustrates how the integrand behaves along the recursion.
\subsection{Recursion on tail varieties}\label{sec:recursion_tail_varieties} 
Section \ref{sec:NJ_fiber_volume} provides us all the ingredients to apply \eqref{eq:coarea_general_clean} recursively to tail varieties. Define the recursive function $F_r: \Herm_{>0}(D_r) \to \R$ given by
\begin{equation}\label{eq:F_recursion}
    F_r(\Sigma)
:=
\det(\Sigma)^{m_r}
\int_{\St(D_r,n_r)}
F_{r-1}\!\left(
L_{M^r}(\Sigma)
\right)\,d\mu(M^r),
\end{equation}
with the linear function $L_{M^r}(\Sigma)
:=
\sum_{j_r=1}^{d_r}
M^r_{j_r}\Sigma (M_{j_r}^r)^*$ and initial steps $F_0 =1,~F_1(\Sigma) = \det(\Sigma)^{m_1}$.
\begin{proposition}\label{prop:recursive_coarea_formula}
The following recursive relation holds
\begin{align}\label{eq:integral_recursion}
\int_{X_r}F_{r-1}(\Sigma_r([T_r]))\,d\vol_{X_r}([T_r])
&=
\vol(\Gr(D_r, n_r))
\int_{X_{r+1}}F_r(\Sigma_{r+1}([T_{r+1}]))\,d\vol_{X_{r+1}}([T_{r+1}]).
\end{align}
\end{proposition}
\begin{proof}
We apply the smooth coarea formula to \(\Psi_r:Y_r\rightarrow X_r\) for the function $h_r:X_r\to \R$ defined as 
\begin{equation*}
h_r([T_r]):= \frac{F_{r-1}(\Sigma_r([T_r]))}{J_r([T_{r}])}\,.
\end{equation*}
By Proposition \ref{prop:fiber_volume_Psi_r}, the right-hand side of \eqref{eq:coarea_general_clean} reads
\begin{align*}
\int_{X_r}h_r([T_r])\,\vol(\Psi_r^{-1}([T_r]))\,d\vol_{X_r}([T_r]) &= \int_{X_r}
\frac{F_{r-1}(\Sigma_r([T_r]))}{ J_r([T_{r}])}
\cdot
\vol(\U(D_r))\,J_r([T_{r}])
\,d\vol_{X_r}([T_r]) \\
&= \vol(\U(D_r))
\int_{X_r}F_{r-1}(\Sigma_r([T_r]))\,d\vol_{X_r}([T_r]).
\end{align*}
Now consider the left-hand side of \eqref{eq:coarea_general_clean} given by
\begin{align*}
\int_{Y_r}
\NJ(\Psi_r,q_r)\,
h_r(\Psi_r(q_r))
\,dM^r\,d\vol_{X_{r+1}}([T_{r+1}]),
\end{align*}
where $q_r=(M^r,[T_{r+1}])\in Y_r$. Using Corollary \ref{cor:NJ_Psi_r}, and the recursion
formula
\[
\Sigma_r(\Psi_r(M^r,[T_{r+1}]))
=
\sum_{j_r=1}^{d_r}
M^r_{j_r}\Sigma_{r+1}([T_{r+1}])(M^r_{j_r})^* = L_{M^r}(\Sigma_{r+1}([T_{r+1}]),
\]
we obtain
\begin{align}
\notag &\int_{Y_r}
\det(\Sigma_{r+1}([T_{r+1}]))^{m_r}
J_r(\Psi_r(M^r,[T_{r+1}]))
\frac{
F_{r-1}\!\left(L_{M^r}(\Sigma_{r+1}([T_{r+1}]))
\right)
}{
J_r(\Psi_r(M^r,[T_{r+1}]))
}
\,dM^r\,d\vol_{X_{r+1}}([T_{r+1}]) \\ \label{eq:rhs_coarea}
&\qquad=
\int_{Y_r}
\det(\Sigma_{r+1}([T_{r+1}]))^{m_r}
F_{r-1}\!\left(L_{M^r}(\Sigma_{r+1}([T_{r+1}]))
\right)
\,dM^r\,d\vol_{X_{r+1}}([T_{r+1}]).
\end{align}
Recall that the Stiefel measure can be normalized as $dM^r=\vol(\St(D_r,n_r))\,d\mu(M^r)$. Using \eqref{eq:F_recursion}, \eqref{eq:rhs_coarea} can be written as
\begin{align*}
\vol(\St(D_r,D_{r-1}d_r))
\int_{X_{r+1}}
F_r(\Sigma_{r+1}([T_{r+1}]))\,d\vol_{X_{r+1}}([T_{r+1}]).
\end{align*}
Comparing the obtained left and right sides of \eqref{eq:coarea_general_clean}, from \eqref{eq:vol_gr} we get the recursive integral expression.
\end{proof}
\noindent
As a corollary of Proposition \ref{prop:recursive_coarea_formula}, we obtain an expression for the volume of TT varieties.
\begin{corollary}\label{cor:full-recursive-volume}
The volume of the tensor train variety of signature $(D, d)$ is
\begin{equation}\label{eq:full-recursive-volume}
\vol(V_{D,d})
=
\left(\prod_{r=1}^{N-1}\vol\bigl(\Gr(D_r,n_r)\bigr)\right)
\int_{X_N}F_{N-1}(CC^*)\,d\vol_{X_N}([C]).
\end{equation}
\end{corollary}

\begin{proof}
Recall that $V_{D,d}^= = X_1$. We start from \eqref{eq:integral_recursion} with $r=1$ and $F_0 = 1$:
\[
\vol(X_1)
=
\vol\bigl(\Gr(D_1,d_1)\bigr)
\int_{X_2}F_1(\Sigma_2([T_2]))\,d\vol_{X_2}([T_2]).
\]
Applying \eqref{eq:integral_recursion} to the integral of the right-hand side successively for \(r=2,3,\dots,N-1\)
we get
\[
\vol(V^{=}_{D,d}) =\vol(X_1)
=
\left(\prod_{r=1}^{N-1}\vol\bigl(\Gr(D_r,n_r)\bigr)\right)
\int_{X_N}F_{N-1}(\Sigma_N([C]))\,d\vol_{X_N}([C]).
\]
Since the complement of $V_{D,d}^{=}$ in $V_{D,d}$ has measure zero and \(\Sigma_N([C])=CC^*\), we obtain \eqref{eq:full-recursive-volume}.
\end{proof}
\noindent
The computation of the volume of the TT variety has thus been reduced to an integral over the locus \(X_N\) of full rank matrices.

\subsection{Recursion on tail polynomials}\label{sec:integration}
In order to perform the integration in \eqref{eq:full-recursive-volume}, we need to understand the integrand. We claim that
\[
P_{D,d}(x,y):=F_{N-1}(CC^*)
=
\sum_{\alpha,\beta} c_{\alpha,\beta}\,x^\alpha y^\beta
\]
is a polynomial in the entries $x$ of \(C\) and their formal conjugates \(y\).

\begin{proposition}\label{prop:polynomial_C}
The function $P_{D,d}(x,y)$ is a bihomogeneous polynomial of bidegree \((a,a)\), where
\[
a=\sum_{r=1}^{N-1}D_rm_r.
\]
\end{proposition}

\begin{proof}
We first show by induction on \(r\) that \(F_r(\Sigma)\) is a homogeneous polynomial in the entries of
\(\Sigma\in \Herm_{>0}(D_r)\), of total degree
\begin{equation}\label{eq:a_r}
a_r:=\sum_{i=1}^r D_i m_i.
\end{equation}
For \(r=1\), we have $F_1(\Sigma)=\det(\Sigma)^{m_1}$,
which is a homogeneous polynomial of degree \(D_1m_1=a_1\) in the entries of \(\Sigma\).
Now assume that \(F_{r-1}(\Sigma')\) is a homogeneous polynomial of degree \(a_{r-1}\) in the entries of
\(\Sigma'\in \Herm_{>0}(D_{r-1})\). Recall \eqref{eq:F_recursion} and the map $L_{M^r}(\Sigma)$, which is linear in the entries of \(\Sigma\in \Herm_{>0}(D_r)\). Hence $F_{r-1}(
L_{M^r}(\Sigma))$
is a homogeneous polynomial of degree \(a_{r-1}\) in the entries of \(\Sigma\). Integrating over
the Stiefel manifold averages the coefficients and therefore does not change the degree in the entries of $\Sigma$. Multiplying by \(\det(\Sigma)^{m_r}\), which has degree \(D_rm_r\),
shows that \(F_r\) is a homogeneous polynomial of degree $a_{r-1}+D_rm_r=a_r$.
Applying this for \(r=N-1\), it follows that \(F_{N-1}(\Sigma)\) is a homogeneous polynomial of degree $a$ in the entries of \(\Sigma\). The entries of
\(\Sigma=CC^*\) are
\[
\Sigma_{ab}=\sum_{k=1}^{d_N}x_{ak}y_{bk},
\]
hence each entry of \(\Sigma\) is bihomogeneous of bidegree \((1,1)\) in \((x,y)\). It follows that $P_{D,d}(x,y)$ is bihomogeneous of bidegree \((a,a)\).
\end{proof}

We now show that the last tail factor $F_{N-1}(CC^*)$ can be expressed as a sum of Schur polynomials. This will allow us to integrate \eqref{eq:full-recursive-volume} in Section \ref{sec:degree}. \\

\begin{lemma}\label{lemma:F_r_invariance}
    The polynomial $F_r(\Sigma)$ is $\U(D_r)$-invariant, i.e., $F_r(U\Sigma U^*)=F_r(\Sigma)$ for every $U \in \U(D_r)$.
\end{lemma}
\begin{proof}
    For every $U \in \U(D_r)$ we have $L_{M^r}(U\Sigma U^*) = L_{M^rU}(\Sigma)$. Hence 
    \begin{align*}
        F_r(U\Sigma U^*) =& \ \det(U\Sigma U^*)^{m_r}\int_{\St(D_r,n_r)}
F_{r-1}\!\left(
L_{M^r U}(\Sigma)
\right)\,d\mu(M^r) \\=& \ \det(\Sigma)^{m_r}\int_{\St(D_r,n_r)}
F_{r-1}\!\left(
L_{M^r}(\Sigma)
\right)\,d\mu(M^r) = F_r(\Sigma),
    \end{align*}
by the right-invariance of the Haar measure under the transformation $M^r \mapsto M^rU$. 
\end{proof}

By Lemma \ref{lemma:F_r_invariance}, $F_r(\Sigma)$ is a symmetric polynomial of degree $a_r$ in the eigenvalues $t:=(t_1,\dots,t_{D_r})$ of \(\Sigma\). Therefore, it admits an expansion in terms of Schur polynomials \cite{Macdonald_1995}. For a partition $\mu \vdash a_r$ of length $\ell(\mu) \leq D_{r}$ the corresponding Schur polynomial reads
    \begin{equation}\label{eq:schur}
        s_{\mu}(t) := \frac{\det(t_i^{\mu_j-D_r-j})_{i,j=1}^{D_r}}{\det(t_i^{D_r-j})_{i,j=1}^{D_r}}.
    \end{equation}
These form a basis for the space of symmetric polynomials of degree $a_r$ in $D_r$ variables. We denote by $s_{\mu}(\Sigma)$ the Schur polynomial $s_{\mu}$ evaluated at the eigenvalues of $\Sigma$. 
\\~\\
Consider the Schur expansion
\begin{align}\label{eq:c}
    F_{r-1}(L_{M^r}(\Sigma)) = \sum_{\substack{\lambda\vdash a_{r-1}\\ \ell(\lambda)\le D_{r-1}}} c_{\lambda}\,s_{\lambda}(L_{M^r}(\Sigma)) \, ,
\end{align}
where $c_{\lambda} \in \C$. Recall the integration of $F_{r-1}$ in the definition \eqref{eq:F_recursion} of $F_r$. Since the integral is linear, we introduce 
\begin{equation}\label{eq:H_function}
    H_\lambda^{(r)}(\Sigma)
:=
\int_{\St(D_r,n_r)}
s_\lambda\!\left(L_{M^r}(\Sigma)\right)
\,d\mu(M^r),
\end{equation}
for any partition $\lambda$ of $a_{r-1}$ with length $\ell(\lambda) \leq D_{r-1}$. 
Since \(H_\lambda^{(r)}\) is homogeneous of degree $|\lambda|=a_{r-1}$ and
\(\U(D_r)\)-invariant, it also admits a Schur expansion
\begin{equation}\label{eq:H_function_exp}
    H_\lambda^{(r)}(\Sigma)
=
\sum_{\substack{\mu\vdash a_{r-1}\\ \ell(\mu)\le D_r}}
h_{\lambda,\mu}^{(r)}s_\mu(\Sigma).
\end{equation}

\noindent In order to determine the coefficients $h_{\lambda,\mu}^{(r)}$ we use the \emph{Stiefel-Weingarten formula} \cite{Collins_2006,Collins_2017,Collins_2022}. This allows to integrate polynomials over Stiefel manifolds.
\begin{proposition}[Stiefel--Weingarten formula, Corollary 2.4 \cite{Collins_2006}]\label{prop:stiefel_weingarten}
Let \(M^r\in \St(D_r,n_r)\). Then for \(a_{r-1}\ge 1\),
\begin{equation}\label{eq:stiefel_weingarten_formula}
\int_{\St(D_r,n_r)}
\prod_{s=1}^{a_{r-1}} M^r_{i_s a_s}\,
\overline{M^r_{i'_s c_s}}
\,d\mu(M^r)
=
\sum_{\sigma,\tau\in S_{a_{r-1}}}
\left(\prod_{s=1}^{a_{r-1}} \delta_{i_s,i'_{\sigma(s)}}\right)
\left(\prod_{s=1}^{a_{r-1}} \delta_{a_s,c_{\tau(s)}}\right)
\mathrm{Wg}_{n_r}(\sigma^{-1}\tau),
\end{equation}
where $S_{a_{r-1}}$ is the symmetric group and \(\mathrm{Wg}_{n_r}\) is the unitary
Weingarten function.
\end{proposition}
\begin{theorem}[Schur--Weingarten coefficients]\label{thm:H_coefficients}
For fixed $r$, consider two partitions \(\lambda,\mu\vdash a_{r-1}\) with $\ell(\lambda)\le D_{r-1}$ and $\ell(\mu)\le D_r$.
Then the Schur coefficients in \eqref{eq:H_function_exp} have the form
\begin{equation}\label{eq:h_coeff}
    h_{\lambda,\mu}^{(r)}
=
\frac{s_\lambda(1^{D_{r-1}})}
{f^\lambda C_\mu(n_r)}
\sum_{\kappa\vdash a_{r-1}}
|\mathcal C_\kappa|\,
d_r^{\ell(\kappa)}
\chi^\lambda(\kappa)\chi^\mu(\kappa).
\end{equation}
Here $C_\mu(n_r)
:=
\prod_{(i,j)\in\mu}(n_r+j-i)$, \(\mathcal C_\kappa\) denotes the conjugacy class in \(S_{a_{r-1}}\) of cycle type \(\kappa\), \(f^\lambda\) is the dimension of the irreducible \(S_{a_{r-1}}\)-representation
indexed by \(\lambda\) and $\chi^\lambda(\kappa)$ is its character on \(\mathcal C_\kappa\).
\end{theorem}

\begin{proof}
Write the row index of \(M^r\in\St(D_r, n_r) \subset \C^{n_r\times D_r}\) as a pair $(\alpha,j_r)$ with $\alpha=1,\dots,D_{r-1}$ and
$j_r=1,\dots,d_r$.
Here \(\alpha\) is the row index inside a block and \(j_r\) is the block label.
\\
The entries of $L_{M^r}(\Sigma)=\sum_{j_r=1}^{d_r} M^r_{j_r}\Sigma (M_{j_r}^r)^*$ are $(L_{M^r}(\Sigma))_{\alpha\beta}
=
\sum_{j_r=1}^{d_r}
\sum_{\nu, \eta=1}^{D_r}
M^r_{(\alpha,b),\nu}\,
\Sigma_{\nu \eta}\,
\overline{M^r_{(\beta,b),\eta}}$.
\\
For a cycle type $\kappa=(\kappa_1,\dots,\kappa_m)$ $\ell(\kappa)=m$ the number of cycles. If $\pi \in S_{a_{r-1}}$ has cycle type $\kappa$, we define $\ell(\pi):=\ell(\kappa)$. By the Frobenius character formula \cite{Stanley_2024}, 
\begin{equation}\label{eq:frobenius}
s_\lambda(L_{M^r}(\Sigma))
=
\frac1{(a_{r-1})!}
\sum_{\pi\in S_{a_{r-1}}}
\chi^\lambda(\pi)p_\pi(L_{M^r}(\Sigma)),
\end{equation}
where, for \(\pi\in S_{a_{r-1}}\), $p_{\pi}(L_{M^r}(\Sigma))$ is the power-sum (or trace expansion)
\begin{equation*}
    p_\pi(L_{M^r}(\Sigma))
= \prod_{\text{cycles c of $\pi$}}\tr(L_{M^r}(\Sigma)^{|c|}) 
=
\sum_{\gamma_1, \dots, \gamma_{a_{r-1}}=1}^{D_{r-1}}
\prod_{s=1}^{a_{r-1}}
L_{M^r}(\Sigma)_{\gamma_s,\gamma_{\pi(s)}}.
\end{equation*}
For the reader's convenience, we abbreviate the sum over indices $\gamma_1, \dots, \gamma_{a_{r-1}}$ for each $\gamma_i$ ranging from $1$ to $D_{r-1}$ with $\sum_{\gamma=1}^{D_{r-1}}$. We use this shorthand for all following multiindices in this proof.
\\
Therefore,
\[
H_\lambda^{(r)}(\Sigma)
=
\frac1{a_{r-1}!}
\sum_{\pi\in S_{a_{r-1}}}
\chi^\lambda(\pi)
\int_{\St(D_r,n_r)}
p_\pi(L_{M^r}(\Sigma))\,d\mu(M^r).
\]
\\
Using the entries formula for \(L_{M^r}(\Sigma)\), the integrand becomes
\begin{align*}
    p_\pi(L_{M^r}(\Sigma))
&=
\sum_{\alpha=1}^{D_{r-1}}\sum_{j_r=1}^{d_r}\sum_{\nu, \eta=1}^{D_r}
\prod_{s=1}^{a_{r-1}}
M^r_{(\alpha_s,(j_r)_s),\nu_s}\,
\Sigma_{\nu_s\eta_s}\,
\overline{M^r_{(\alpha_{\pi(s)},(j_r)_s),\eta_s}} \\
&=
\sum_{\alpha=1}^{D_{r-1}}\sum_{j_r=1}^{d_r}\sum_{\nu, \eta=1}^{D_r}
\left(
\prod_{s=1}^{a_{r-1}} \Sigma_{\nu_s \eta_s}
\right)
\left(
\prod_{s=1}^{a_{r-1}}
M^r_{i_s,\nu_s}\,
\overline{M^r_{i'_s,\eta_s}}
\right),
\end{align*}
where we introduced combined indices $i_s =(\alpha_s,(j_r)_s)$ and $i_s'=(\alpha_{\pi(s)},(j_r)_s)$. Note that each index $\alpha, j_r, \nu, \eta$ is meant as multiindex, as introduced above.
\\
Since $\Sigma$ is fixed, we apply Proposition \ref{prop:stiefel_weingarten} and get
\begin{align*}
&\int_{\St(D_r,n_r)}
p_\pi(L_{M^r}(\Sigma))\,d\mu(M^r)  \\ 
&\qquad\qquad= \sum_{\alpha=1}^{D_{r-1}}\sum_{j_r=1}^{d_r}\sum_{\nu, \eta=1}^{D_r}
\left(
\prod_{s=1}^{a_{r-1}} \Sigma_{\nu_s \eta_s}
\right)
\int_{\St(D_r,n_r)}
\prod_{s=1}^{a_{r-1}}
M^r_{i_s,\nu_s}\,
\overline{M^r_{i'_s,\eta_s}}
\,d\mu(M^r)\\
&\qquad\qquad=
\sum_{\sigma,\tau\in S_{a_{r-1}}}
\Wg_{n_r}(\sigma^{-1}\tau)
\sum_{\alpha=1}^{D_{r-1}}\sum_{j_r=1}^{d_r}\sum_{\nu, \eta=1}^{D_r}
\left(
\prod_{s=1}^{a_{r-1}}
\delta_{i_s,i'_{\sigma(s)}}
\right)
\left(
\prod_{s=1}^{a_{r-1}}
\delta_{\nu_s,\eta_{\tau(s)}}
\right)
\left(
\prod_{s=1}^{a_{r-1}} \Sigma_{\nu_s\eta_s}
\right).
\end{align*}
\\
First, we consider the $\delta$-function $\delta_{i_s,i'_{\sigma(s)}}
=
\delta_{(\alpha_s,b_s),(\alpha_{\pi(\sigma(s))},b_{\sigma(s)})} = \delta_{\alpha_s,\alpha_{\pi(\sigma(s))}}\delta_{b_s,b_{\sigma(s)}}$. These are non-zero if and only if
$\alpha_s=\alpha_{\pi\sigma(s)}$ and $b_s=b_{\sigma(s)}$. The first condition identifies two \(\alpha\)-indices along the cycles of
\(\pi\sigma\) giving non-zero terms. For example, let $\pi\sigma = (134)(25)$, such that $\alpha_s = \alpha_{\pi\sigma(s)}$ implies that $\alpha_{1} =\alpha_{3} =\alpha_{4}$ and $\alpha_{2} = \alpha_{5}$. Therefore, the sum over all \(\alpha_1,\dots,\alpha_{a_{r-1}}\) gives
one free choice in \(\{1,\dots,D_{r-1}\}\) for each cycle of \(\pi\sigma\). The same argument holds for the condition $b_s=b_{\sigma(s)}$, hence we can write
\begin{equation*}
    \sum_{\alpha_1,\dots,\alpha_{a_{r-1}}=1}^{D_{r-1}}
\prod_{s=1}^{a_{r-1}}
\delta_{\alpha_s,\alpha_{\pi\sigma(s)}}
=
D_{r-1}^{\ell(\pi\sigma)}, \quad
\sum_{b_1,\dots,b_{a_{r-1}} =1}^{d_r}
\prod_{s=1}^{a_{r-1}}
\delta_{b_s,b_{\sigma(s)}}
=
d_r^{\ell(\sigma)}.
\end{equation*}
\\
It remains to simplify the summation over $\nu, \eta$. From $\prod_{s=1}^{a_{r-1}}\delta_{\nu_s,\eta_{\tau(s)}}$ we get non-zero contributions if and only if $\eta_s=\nu_{\tau^{-1}(s)}$ for every $
s=1,\dots,a_{r-1}$. Therefore,
\[
\sum_{\nu, \eta=1}^{D_r}
\prod_{s=1}^{a_{r-1}}
\Sigma_{\nu_s\eta_s}
\prod_{s=1}^{a_{r-1}}
\delta_{\nu_s,\eta_{\tau(s)}}
=
\sum_{\nu_1,\dots,\nu_{a_{r-1}}=1}^{D_{r-1}}
\prod_{s=1}^{a_{r-1}}
\Sigma_{\nu_s,\nu_{\tau^{-1}(s)}} = p_{\tau^{-1}}(\Sigma) = p_\tau(\Sigma) \, ,
\]
since \(\tau\) and
\(\tau^{-1}\) have the same cycles. Combining the above evaluations of $\delta$-functions, we obtain
\[
\int_{\St(D_r,n_r)}
p_\pi(L_{M^r}(\Sigma))\,d\mu(M^r)
=
\sum_{\sigma,\tau\in S_{a_{r-1}}}
D_{r-1}^{\ell(\pi\sigma)}
d_r^{\ell(\sigma)}
\Wg_{n_r}(\sigma^{-1}\tau)
p_\tau(\Sigma).
\]
\\
Substituting this into the Frobenius formula \eqref{eq:frobenius} gives
\[
H_\lambda^{(r)}(\Sigma)
=
\sum_{\sigma,\tau\in S_{a_{r-1}}}
\left[
\frac1{a_{r-1}!}
\sum_{\pi\in S_{a_{r-1}}}
\chi^\lambda(\pi)
D_{r-1}^{\ell(\pi\sigma)}
\right]
d_r^{\ell(\sigma)}
\Wg_{n_r}(\sigma^{-1}\tau)
p_\tau(\Sigma).
\]
We now simplify further the term in brackets. For any \(\rho\in S_{a_{r-1}}\) and the $D_{r-1}\times D_{r-1}$ identity matrix $1^{D_{r-1}}$ we write the power sum as $p_\rho(1^{D_{r-1}})
=
\prod_{\text{cycles }c\text{ of }\rho}
\tr((1^{D_{r-1}})^{|c|})$. But $\tr((1^{D_{r-1}})^{|c|})
=
1^{|c|}+\cdots+1^{|c|}
=
D_{r-1}$. Therefore, $p_\rho(1^{D_{r-1}})
=
D_{r-1}^{\ell(\rho)}$.
Using the inverse Frobenius relation \cite{Stanley_2024}
\begin{equation}\label{eq:inverse_frobenius}
p_\rho
=
\sum_{\eta\vdash |\rho|}
\chi^\eta(\rho)s_\eta,
\end{equation}
we get $D_{r-1}^{\ell(\rho)}
=
p_\rho(1^{D_{r-1}})
=
\sum_{\eta\vdash |\rho|}
\chi^\eta(\rho)s_\eta(1^{D_{r-1}})$. Applying it to \(\rho=\pi\sigma\), this gives $D_{r-1}^{\ell(\pi\sigma)}
=
\sum_{\eta\vdash a_{r-1}}
\chi^\eta(\pi\sigma)s_\eta(1^{D_{r-1}})$.
Hence
\[
\frac1{a_{r-1}!}
\sum_{\pi\in S_{a_{r-1}}}
\chi^\lambda(\pi)
D_{r-1}^{\ell(\pi\sigma)}
=
\frac1{a_{r-1}!}
\sum_{\eta\vdash a_{r-1}}
s_\eta(1^{D_{r-1}})
\sum_{\pi\in S_{a_{r-1}}}
\chi^\lambda(\pi)\chi^\eta(\pi\sigma).
\]
We use the character orthogonality relation \cite{Stanley_2024},
\begin{equation}\label{eq:ortho}
    \sum_{\pi\in S_{a_{r-1}}}
\chi^\lambda(\pi)\chi^\eta(\pi\sigma)
=
\frac{a_{r-1}!}{f^\lambda}
\delta_{\lambda,\eta}\chi^\lambda(\sigma),
\end{equation}
to simplify the bracketed term to
\[
\frac1{a_{r-1}!}
\sum_{\pi\in S_{a_{r-1}}}
\chi^\lambda(\pi)
D_{r-1}^{\ell(\pi\sigma)}
=
\frac{s_\lambda(1^{D_{r-1}})}{f^\lambda}
\chi^\lambda(\sigma).
\]
Substituting this into \eqref{eq:H_function} we obtain
\[
H_\lambda^{(r)}(\Sigma)
=
\frac{s_\lambda(1^{D_{r-1}})}{f^\lambda}
\sum_{\sigma,\tau\in S_{a_{r-1}}}
d_r^{\ell(\sigma)}
\chi^\lambda(\sigma)
\Wg_{n_r}(\sigma^{-1}\tau)
p_\tau(\Sigma).
\]
We need to express this function in terms of Schur polynomials. Therefore, we convert \(p_\tau(\Sigma)\) back to Schur functions using the inverse Frobenius formula \eqref{eq:inverse_frobenius}. The coefficient of \(s_\mu(\Sigma)\) is
\[
h_{\lambda,\mu}^{(r)}
=
\frac{s_\lambda(1^{D_r-1})}{f^\lambda}
\sum_{\sigma,\tau\in S_{a_{r-1}}}
d_r^{\ell(\sigma)}
\chi^\lambda(\sigma)
\Wg_{n_r}(\sigma^{-1}\tau)
\chi^\mu(\tau).
\]
Next, we use the Weingarten character identity
\begin{equation}\label{eq:weingarten_id}
    \sum_{\tau\in S_{a_{r-1}}}
    \Wg_{n_r}(\sigma^{-1}\tau)\chi^\mu(\tau)
    =
    \frac{\chi^\mu(\sigma)}{C_\mu(n_r)}.
\end{equation}
This can be derived by noting that the unitary Weingarten function has the
character expansion \cite{Collins_2006,Collins_2017,Collins_2022}
\[
\Wg_{n_r}(\rho)
=
\frac1{a_{r-1}!}
\sum_{\substack{\eta\vdash a_{r-1}\\ \ell(\eta)\le n_r}}
\frac{f^\eta}{C_\eta(n_r)}
\chi^\eta(\rho),
\]
such that
\[
\sum_{\tau\in S_{a_{r-1}}}
\Wg_{n_r}(\sigma^{-1}\tau)\chi^\mu(\tau)
=
\frac1{a_{r-1}!}
\sum_{\substack{\eta\vdash a_{r-1}\\ \ell(\eta)\le n_r}}
\frac{f^\eta}{C_\eta(n_r)}
\sum_{\tau\in S_{a_{r-1}}}
\chi^\eta(\sigma^{-1}\tau)\chi^\mu(\tau).
\]
The result \eqref{eq:weingarten_id} follows immediately by using character orthogonality \eqref{eq:ortho}. Using this, we get
\begin{equation}\label{eq:h_coeff_proof}
    h_{\lambda,\mu}^{(r)}
=
\frac{s_\lambda(1^{D_r-1})}
{f^\lambda C_\mu(n_r)}
\sum_{\sigma\in S_{a_{r-1}}}
d_r^{\ell(\sigma)}
\chi^\lambda(\sigma)
\chi^\mu(\sigma).
\end{equation}
Since $\ell(\sigma)=\ell(\kappa)$ and $\chi^\lambda(\sigma) = \chi^\lambda(\kappa)$ are both constant on the conjugacy class of $\sigma$, the summands in \eqref{eq:h_coeff_proof} only depend on its cycle type $\kappa$. Let \(m_q(\kappa)\) be the number of cycles of length $q$ for a permutation in $S_{a_{r-1}}$ with cycle type $\kappa$. The number of  permutations of cycle
type $\kappa$ is
\begin{equation}\label{eq:C_k}
    |\mathcal C_\kappa|
=
\frac{a_{r-1}!}{z_\kappa},
\qquad
z_\kappa
=
\prod_{q\ge 1}q^{m_q(\kappa)}m_q(\kappa)!.
\end{equation}
Using the above observation, \eqref{eq:h_coeff_proof} can be written as
\[
h_{\lambda,\mu}^{(r)}
=
\frac{s_\lambda(1^{D_{r-1}})}
{f^\lambda C_\mu(n_r)}
\sum_{\kappa\vdash a_{r-1}}
|\mathcal C_\kappa|\,
d_r^{\ell(\kappa)}
\chi^\lambda(\kappa)\chi^\mu(\kappa). \qedhere
\]
\end{proof}
\begin{remark}
Realizing that the sum in \eqref{eq:h_coeff_proof} only depends on cycle types is computationally very important. Instead of computing characters of all $a_{r-1}!$ permutations we only need to consider the possible partitions $\kappa \vdash a_{r-1}$. To illustrate the difference in complexity, we compare in Table \ref{tab:comparison} how the number of partitions $P(a_{r-1})$ of $a_{r-1}$ and the number of permutations in $S_{a_{r-1}}$ scale. 
\begin{table}[H]
\begin{center}
\begin{tabular}{|c r r|} 
\hline
 $a_{r-1}$ & $a_{r-1}!$ & $P(a_{r-1})$ \\ [0.5ex] 
 \hline
 3  & 6 & 3 \\ 
 4  & 24 & 5 \\
 5  & 120 & 7 \\
 6  & 720 & 11 \\
 7  & 5040 & 15 \\
 8  & 40320 & 22 \\
 10 & 3628800 & 42 \\
 15 & 1307674368000 & 176 \\
 20 & 2432902008176640000 & 627 \\ [1ex] 
 \hline
\end{tabular}
\caption{Comparison of factorial growth and partition numbers.}
\label{tab:comparison}
\end{center}
\end{table}
\end{remark}

\noindent
We now employ Theorem \ref{thm:H_coefficients} to obtain recursively the Schur expansion of the last tail factor. Throughout we will use the following Lemma. 
\begin{lemma}\label{lemma:det_shift}
    Let $\Sigma$ be a $D\times D$ matrix with eigenvalues $t:=(t_1, ..., t_D)$ and $m \in \N$. For a partition $\lambda \vdash a$ with length $\ell(\lambda) \leq D$ we have
    \begin{equation*}
        \det(\Sigma)^ms_\lambda(\Sigma) = s_{\lambda + (m^D)},
    \end{equation*}
    where $(m^D) = (m, ..., m)$, $\lambda + (m^D) = (\lambda_1 + m, ..., \lambda_D +m)$ and $\lambda$ is padded with zeros to length $D$.
\end{lemma}
\begin{proof}We can write $\det(\Sigma) = e_D(t_1, ..., t_D)$, where $e_i$ for $i\in\N$ is the elementary symmetric polynomial. 
For the partition $\lambda$ we define $\mu := \lambda + (m^D)$. The conjugate partition $\lambda'$ has $\lambda_1$ rows, that is, the length is $\ell(\lambda') = \lambda_1$.
We employ the second Jacobi-Trudi formula \cite{Fulton_1991}
\begin{equation*}
    s_\mu = \det(e_{\mu'_i+j-i})_{i,j=1}^{\mu_1}.
\end{equation*}
Note that $\mu_k' = \#\{i\,:\,\mu_i \geq k\}$. For $1\leq k \leq m$, we have $\mu_i = \lambda_i + m \geq m \geq k$, hence $\mu_k' = D$. Therefore, the conjugate partition to $\mu$ is $\mu' = (D, ..., D, \lambda_1', ..., \lambda_{\lambda_1}')$.
Let us write the matrix of the determinant explicitly:
\begin{equation*}
(e_{\mu'_i+j-i})_{i,j=1}^{\mu_1} =
    \begin{bmatrix}
        A & 0 \\
        * & B
    \end{bmatrix},
\end{equation*}
where $A$ is a $m\times m$ and $B$ a $(\mu_1-m)\times (\mu_1-m)$ matrix. The lower triangular structure comes from the fact that for $i \leq m$ and $m < j \leq \mu_1$ we have $j-i \geq 1$ and $\mu_i' = D$. Therefore, the index $\mu'_i+j-i > D$. But for any elementary symmetric function $e_i(t_1, ..., t_D)$ we have $e_i = 0$ if $i > D$. By the same argument, the block $A$ is also  lower triangular with diagonal elements equal to $e_D$, hence $\det(A) = (e_D)^m$. For the block $B$ note that $m < i, j \leq \lambda_1 + m$, therefore $B = (e_{\lambda'_i+j-i})_{i,j=1}^{\lambda_1}$ since $\mu_{i+m}' = \lambda_i'$. It follows that 
\begin{equation*}
    s_{\lambda + (m^D)}=\det(e_{\mu'_i+j-i})_{i,j=1}^{\mu_1} = \det(A)\det(B) = (e_D)^ms_\lambda(\Sigma). \qedhere
\end{equation*}
\end{proof}
Lemma \ref{lemma:det_shift} establishes how the determinant factor in \eqref{eq:F_recursion} shifts the partition of Schur functions. We can finally write the recursion explicitly:
\paragraph{Step $r=0,1$} The recursion starts trivially with $F_0=1$. Also $F_1(\Sigma)
=
\det(\Sigma)^{m_1} = s_{(m_1^{D_1})}(\Sigma).
$
\paragraph{Step $r=2$} 
The first nontrivial step is to expand using \eqref{eq:H_function_exp}
\[
H_{(m_1^{D_1})}^{(2)}(\Sigma) = \sum_{\substack{\mu\vdash a_1\\ \ell(\mu)\le D_2}}
h_{(m_1^{D_1}),\mu}^{(2)}
s_\mu(\Sigma).
\]
Together with the determinant factor we have
\[
F_2(\Sigma)
=
\det(\Sigma)^{m_2}H_{(m_1^{D_1})}^{(2)}(\Sigma) = \sum_{\substack{\mu\vdash a_1\\ \ell(\mu)\le D_2}}
h_{(m_1^{D_1}),\mu}^{(2)}
s_{\mu+(m_2^{D_2})}(\Sigma),
\]
where we used the determinantal shift from Lemma \ref{lemma:det_shift}.

\paragraph{Step $r>2$} 
Assume that, for some \(r\ge 2\), we already have
\[
F_{r-1}(\Sigma)
=
\sum_{\substack{\lambda\vdash a_{r-1}\\ \ell(\lambda)\le D_{r-1}}}
c_\lambda^{(r-1)}s_\lambda(\Sigma).
\]
Then by linearity of integration, we obtain the form
\begin{align*}
    \int_{\St(D_r, n_r)}\sum_{\substack{\lambda\vdash a_{r-1}\\ \ell(\lambda)\le D_{r-1}}}
c_\lambda^{(r-1)}s_\lambda(L_{M^r}(\Sigma))d\mu(M^r)
&= \sum_{\substack{\lambda\vdash a_{r-1}\\ \ell(\lambda)\le D_{r-1}}}
c_\lambda^{(r-1)}
H_\lambda^{(r)}(\Sigma) = \sum_{\substack{\mu\vdash a_{r-1}\\ \ell(\mu)\le D_r}}
g_\mu^{(r)}s_\mu(\Sigma)\,,
\end{align*}
where we define
\begin{equation}\label{eq:g_coeff}
    g_\mu^{(r)}
:=
\sum_{\substack{\lambda\vdash a_{r-1}\\ \ell(\lambda)\le D_{r-1}}}
c_\lambda^{(r-1)}h_{\lambda,\mu}^{(r)}
\,.
\end{equation}
\\
Finally, using again the determinantal shift of partitions, we have
\[
F_r(\Sigma)
=
\det(\Sigma)^{m_r}\sum_{\substack{\mu\vdash a_{r-1}\\ \ell(\mu)\le D_r}}
g_\mu^{(r)}s_\mu(\Sigma) = \sum_{\substack{\mu\vdash a_{r-1}\\ \ell(\mu)\le D_r}}
g_\mu^{(r)}
s_{\mu+(m_r^{D_r})}(\Sigma).
\]
We conclude the recursion at $r = N-1$ with the final tail polynomial
\begin{equation}\label{eq:schur_exp_F}
    F_{N-1}(CC^*)
=
\sum_{\substack{\lambda\vdash a_{N-1}\\ \ell(\lambda)\le D_{N-1}}}
c_\lambda^{(N-1)}s_\lambda(\Sigma)
=
\sum_{\substack{\mu\vdash a_{N-2}\\ \ell(\mu)\le D_{N-1}}}
g_\mu^{(N-1)}
s_{\mu+(m_{N-1}^{D_{N-1}})}(CC^*).
\end{equation}

\section{Evaluating the degree}\label{sec:degree}
In Section \ref{sec:integration} we have expanded the last tail factor $F_{N-1}(CC^*)$ in the basis of Schur polynomials. In this section we write $F_{N-1}(CC^*)=P_{D,d}(C)$, emphasizing its polynomiality. It remains to integrate this factor over the tail variety $X_N$ as in Corollary \ref{cor:full-recursive-volume}.
Since the complement of $X_N = \{[C] \in \P(\C^{D_{N-1}\times d_N}):\rank(C) = D_{N-1}\}$ in $\P(\C^{D_{N-1}\times d_N}) =: \P^K$, that is, the rank-deficient matrices, has measure zero, we can equivalently integrate over $\P^K$.
Let us write the integral in \eqref{eq:full-recursive-volume} as
\begin{equation}\label{eq:int_with_exp}
\int_{\P^K}P_{D,d}(C)\,d\vol_{\P^K}([C]) = \vol(\P^K) \Lambda(D, d) 
\end{equation}
with $\Lambda(D, d) = \mathbb{E}_{[C] \sim \mathrm{Unif}(\P^K)}\Big[P_{D,d}(C)\Big]$. We say that a random matrix $C \in \P^K$ is a \emph{standard complex Gaussian matrix}, written $C \sim \mathcal N_{\C}(0,1)$, if its entries are i.i.d. standard complex Gaussians: if $C=(x_{rs})$ and $x_{rs}=\alpha_{rs}+i\beta_{rs}$, then $\alpha_{rs}, \beta_{rs}\sim \mathcal N(0,1/2)$. If $C \sim \mathcal N_{\C}(0,1)$, its Frobenius squared norm is distributed as $\frac{1}{2}\chi^2_{2(K+1)}$, where $\chi^2_{2(K+1)}$ is the $\chi^2$-\emph{distribution} with $2(K+1)$ degrees of freedom. Note that $C/\|C\| \sim \mathrm{Unif}(\mathbb S^{2K+1})$.  
\begin{lemma}\label{lemma:lambda_D}
    The expectation value in $\eqref{eq:int_with_exp}$ is given by
\begin{equation*}
    \Lambda(D, d) = \frac{(D_{N-1}d_N -1)!}{(a + D_{N-1}d_N-1)!}\cdot f(P_{D,d}),
\end{equation*}
where $f(P_{D,d})=\mathbb{E}_{C \sim \mathcal N_{\C}(0,1)}\Big[P_{D,d}(C)\Big]$ and $a:= a_{N-1}$.
\end{lemma}
\begin{proof}
Choose normalized representatives $\|T\| =1$ for $[T]\in \P^K$. Since $\mathbb S^{2K+1}
=
\{T\in \mathbb C^{K+1}:\|T\|=1\}$, we can identify $\P^K \simeq \mathbb S^{2K+1} / \U(1)$. Since $P_{D,d}$ depends only on $CC^*$, it is $\U(1)$-invariant and we can conclude that 
\begin{equation*}
    \underset{[C] \sim \mathrm{Unif}(\P^K)}{\mathbb{E}}\Big[P_{D,d}(C)\Big] = \underset{C \sim \mathrm{Unif}(\mathbb S^{2K+1})}{\mathbb{E}}\Big[P_{D,d}(C)\Big].
\end{equation*}
Now let $C\sim \mathrm{Unif}(\mathbb S^{2K+1})$ and $\rho \sim \frac{1}{2}\chi^2_{2(K+1)}$ be independent. Then $\sqrt{\rho}\cdot C\sim\mathcal N_{\C}(0,1)$ and since $P_{D,d}(C)$ is of degree $a$
\begin{align*}
    \underset{C \sim \mathrm{Unif}(\mathbb S^{2K+1})}{\mathbb{E}}
    \Big[P_{D,d}(C)\Big] = \Big(\underset{\rho \sim \frac{1}{2}\chi^2_{2(K+1)}}{\mathbb{E}}\Big[\rho^a\Big]\Big)^{-1} \,\underset{\substack{C \sim \mathrm{Unif}(\mathbb S^{2K+1}) \\ \rho \sim \frac{1}{2}\chi^2_{2(K+1)}}}{\mathbb{E}}\Big[P_{D,d}(\sqrt{\rho}\cdot C)\Big].
\end{align*}
Since $\mathbb{E}_{\rho \sim \frac{1}{2}\chi^2_{2(K+1)}}\Big[\rho^a\Big] = \frac{(a + K)!}{K!}$, we obtain
\begin{equation*}
    \Lambda(D, d) = \frac{K!}{(a+K)!}\,\underset{C \sim \mathcal{N}_{\C}(0,1)}{\mathbb{E}}\Big[P_{D,d}(C)\Big]\,.   \qedhere
\end{equation*}
\end{proof}
\newpage
\noindent
We can now prove the degree formula in Theorem \ref{thm:degree_formula}.
\begin{proof}Using Lemma \ref{lemma:lambda_D} and \eqref{eq:proj_volume}, the integral \eqref{eq:int_with_exp} becomes
\begin{align}\label{eq:projective_moment_formula}
\int_{\P^K}P_{D,d}(C)\,d\vol_{\P^K}([C]) &= \frac{\pi^K}{K!}\frac{K!}{(a + K)!}\,f(P_{D,d}) = \frac{\pi^K}{(a + K)!}\,f(P_{D,d}),
\end{align}
With \eqref{eq:projective_moment_formula}, the volume in \eqref{eq:full-recursive-volume} is
\begin{equation}\label{eq:volume_formula_f_polynomial}
\vol(V_{D,d})
=
\left(\prod_{r=1}^{N-1}\vol\bigl(\Gr(D_r,n_r)\bigr)\right)
\frac{\pi^K}{(K+a)!}\,f(P_{D,d}).
\end{equation}
Applying  \eqref{eq:kinematic_formula} to $\vol\bigl(\Gr(D_r,n_r)\bigr)$ 
together with $\dim_\C(V_{D,d})=K+a$, we obtain the degree formula
\begin{equation}\label{eq:degree_formula_f_polynomial}
\deg(V_{D,d}) = 
\frac{
\displaystyle\prod_{r=1}^{N-1}
\deg\bigl(\Gr(D_r,n_r)\bigr)
}{
\displaystyle\prod_{r=1}^{N-1}(D_r m_r)!
}
\;f(P_{D,d})\, . 
\end{equation}
\end{proof}
\vspace{-3mm}
\noindent
The degree of Grassmanians is given by \cite[Example 19.14]{harris_1992} 
\begin{equation}\label{eq:gr_degree}
    \deg(\Gr(k, n)) = ((k+1)(n-k))!\prod_{i=0}^{k}\frac{i!}{(n-k+1)!}.
\end{equation}

\subsection{Integration of Schur polynomials}
As seen in \eqref{eq:degree_formula_f_polynomial}, the degree of the TT variety now admits a compact formula, where only the function $f(P_{D,d})$ remains to be evaluated. Recall the Schur expansion in \eqref{eq:schur_exp_F}. By linearity of expectation we can write
\begin{equation}\label{eq:f_explicit}
    f(P_{D,d})=\underset{C \sim \mathcal N_{\C}(0,1)}{\mathbb{E}}\Big[P_{D,d}(C)\Big] = \sum_{\substack{\mu\vdash a_{N-2}\\ \ell(\mu)\le D_{N-1}}}
g_\mu^{(N-1)}\underset{C \sim \mathcal N_{\C}(0,1)}{\mathbb{E}}\Big[s_{\mu+(m_{N-1}^{D_{N-1}})}(CC^*)\Big].
\end{equation}
To evaluate the expectation of Schur polynomials over complex normally distributed matrices $C$, we prove the following result.
\begin{proposition}\label{prop:f_schur}
Let  \(C\in \mathbb C^{D_{N-1}\times d_N}\), \(D_{N-1}\le d_N\), be distributed as $\mathcal N_{\C}(0,1)$. Let
\(\lambda=(\lambda_1,\dots,\lambda_{D_{N-1}})\) be a partition, padded with zeros if
necessary. Then
\begin{equation*}
    \underset{C \sim \mathcal N_{\C}(0,1)}{\mathbb E}\Big[s_\lambda(CC^*)\Big]
=
s_\lambda(1^{D_{N-1}})
\prod_{i=1}^{D_{N-1}}
\frac{(\lambda_i+d_N-i)!}{(d_N-i)!}.
\end{equation*}
\end{proposition}

\begin{proof}
Let \(W=CC^*\) with $C \sim \mathcal N_{\C}(0,1)$. Then \(W\) has the complex central Wishart distribution
\(\mathcal{CW}_{D_{N-1}}(I_{D_{N-1}},d_N)\). Its density is \cite{goodman_1963}
\[
d\nu(W)
\propto e^{-\operatorname{tr}(W)}
\det(W)^{d_N-D_{N-1}}\,dW,
\]
where we omit the normalization factor. Since \(s_\lambda(W)\) is a function of the eigenvalues $t := (t_1,\dots,t_{D_{N-1}})$ of \(W\), we diagonalize \(W=U\operatorname{diag}(t_1,\dots,t_{D_{N-1}})U^*\) for some $U\in\U(D_{N-1})$.
By the Weyl integration formula for Hermitian matrices \cite[Theorem 10.1.1]{Faraut_2008}, diagonalizing $W$ changes the measure density (up to a constant factor) to 
\[
dW
\propto\Delta(t)^2\,d\mu_{\U(D_{N-1})}(U)\,dt_1\cdots dt_{D_{N-1}},
\]
where $\Delta(t)=\prod_{1\le i<j\le D_{N-1}}(t_i-t_j)$ and $d\mu_{\U(D_{N-1})}(U)$ is the normalized Haar measure on unitary matrices. In terms of the eigenvalues, the Wishart density becomes 
\begin{equation*}
    d\nu(W)
\propto \Delta(t)^2\,d\mu_{\U(D_{N-1})}(U)\prod_{i=1}^{D_{N-1}} t_i^{d_N-D_{N-1}}e^{-t_i}\,dt_i.
\end{equation*}
Define
\begin{equation*}
    I_\lambda
\propto
\int_{[0,\infty)^{D_{N-1}}}
s_\lambda(t)\Delta(t)^2
 \prod_{i=1}^{D_{N-1}} x_i^{d_N-D_{N-1}}e^{-t_i}\,dt_i,
\end{equation*}
and let $I_0$ be the normalization constant of the measure. Then $\mathbb E_{C \sim N_{\C}(0,1)}\Big[s_\lambda(CC^*)\Big] = I_\lambda/I_0$.
\\
We evaluate \(I_\lambda\) using Andréief's identity \cite{andreief_1886}:
if \(\phi_1,\dots,\phi_{D_{N-1}}\) and \(\psi_1,\dots,\psi_{D_{N-1}}\) are integrable, then
\begin{equation}\label{eq:andreief}
    \int
\det(\phi_j(t_i))_{i,j=1}^{D_{N-1}}
\det(\psi_k(t_i))_{i,k=1}^{D_{N-1}}
\prod_{i=1}^{D_{N-1}} d\rho(t_i)
=
D_{N-1}!\det\left(
\int \phi_j(t)\psi_k(t)\,d\rho(t)
\right)_{j,k=1}^{D_{N-1}}.
\end{equation}
By Jacobi's bialternant formula \cite{Macdonald_1995}
\begin{align}\label{eq:s_delta}
s_\lambda(t)\Delta(t)
=&
\det\left(t_i^{\lambda_j+D_{N-1}-j}\right)_{i,j=1}^{D_{N-1}}\ ,\\ \label{eq:delta}
\Delta(t)
=&
\det\left(t_i^{D_{N-1}-k}\right)_{i,k=1}^{D_{N-1}}\ ,
\end{align} 
Applying \eqref{eq:andreief} with $d\rho(t)=t^{d_N-D_{N-1}}e^{-t}\,dt$ and identifying $\det(\phi_j(t_i))_{i,j=1}^{D_{N-1}}$ and $\det(\psi_k(t_i))_{i,k=1}^{D_{N-1}}$ with \eqref{eq:s_delta} and \eqref{eq:delta}, respectively,
we obtain
\begin{align*}
    I_\lambda
&=
D_{N-1}!\det\left(
\int_0^\infty
t^{\lambda_j-j+D_{N-1}-k+d_N}e^{-t}\,dt
\right)_{j,k=1}^{D_{N-1}} \\
&= D_{N-1}!\det\left(\Gamma(\lambda_j+d_N+D_{N-1}-j-k+1)\right)_{j,k=1}^{D_{N-1}},
\end{align*}
where we additionally identified the Gamma integral.
Set $g_j:=\lambda_j+d_N-j+1$, such that 
\[
\Gamma(\lambda_j+d_N+D_{N-1}-j-k+1)
=
\Gamma(g_j+D_{N-1}-k)
=
\Gamma(g_j)(g_j)_{D_{N-1}-k},
\]
where $(g_j)_{D_{N-1}-k} = g_j(g_j + 1)\cdots (g_j + D_{N-1}-k -1)$ is the rising factorial.
Therefore
\[
I_\lambda
=
D_{N-1}!
\left(\prod_{j=1}^{D_{N-1}} \Gamma(\lambda_j+d_N-j+1)\right)
\det\left((g_j)_{D_{N-1}-k}\right)_{j,k=1}^{D_{N-1}}.
\]
Since \((g_j)_{D_{N-1}-k}\) is monic of degree \(D_{N-1}-k\), its determinant is the
Vandermonde determinant
\[
\det\left((g_j)_{D_{N-1}-k}\right)_{j,k=1}^{D_{r-1}}
=
\prod_{1\le i<j\le D_{r-1}}(g_i-g_j).
\]
Using $g_i-g_j=\lambda_i-\lambda_j+j-i$, we get
\begin{align*}
    I_\lambda
&=
D_{N-1}!
\prod_{j=1}^{D_{N-1}} \Gamma(\lambda_j+d_N-j+1)
\prod_{1\le i<j\le D_{N-1}}
(\lambda_i-\lambda_j+j-i), \\
I_0
&=
D_{N-1}!
\prod_{j=1}^{D_{N-1}} \Gamma(d_N-j+1)
\prod_{1\le i<j\le D_{N-1}}
(j-i).
\end{align*}
Thus
\[
\frac{I_\lambda}{I_0}
=
\prod_{j=1}^{D_{N-1}}
\frac{\Gamma(\lambda_j+d_N-j+1)}{\Gamma(d_N-j+1)}
\prod_{1\le i<j\le D_{N-1}}
\frac{\lambda_i-\lambda_j+j-i}{j-i}.
\]
For the Schur polynomial evaluated at the $D_{N-1}\times D_{N-1}$ identity matrix we have \cite{Fulton_1991}
\begin{equation}\label{eq:schur_identity}
s_\lambda(1^{D_{N-1}})
=
\prod_{1\le i<j\le D_{N-1}}
\frac{\lambda_i-\lambda_j+j-i}{j-i}.
\end{equation}
Writing the Gamma functions explicitly we get
\[
\underset{C \sim \mathcal N_{\C}(0,1)}{\mathbb{E}}\Big[s_{\lambda}(CC^*)\Big]
=
s_\lambda(1^{D_{N-1}})
\prod_{j=1}^{D_{N-1}}
\frac{(\lambda_j+d_N-j)!}{(d_N-j)!}. \qedhere
\]
\end{proof}
\noindent
We conclude the derivation by combining \eqref{eq:f_explicit} and \eqref{eq:g_coeff}:
\begin{equation}\label{eq:super_long_f}
        f(P_{D,d}) = \sum_{\substack{\mu, \lambda\vdash a_{N-2}\\ \ell(\mu)\le D_{N-1} \\
        \ell(\lambda)\le D_{N-2}}}
c_\lambda^{(N-2)}h_{\lambda,\mu}^{(N-1)}s_{\mu+(m_{N-1}^{D_{N-1}})}(1^{D_{N-1}})
\prod_{j=1}^{D_{N-1}}
\frac{((\mu+(m_{N-1}^{D_{N-1}}))_j+d_N-j)!}{(d_N-j)!}.
\end{equation}
For the reader's convenience, let us recall the constituents of this expression: the sum runs over partitions $\lambda, \mu$ of an integer $a_{N-2}$ (see \eqref{eq:a_r}) of different lengths. The coefficients $c_\lambda^{(N-2)}$ are rational numbers coming from the second-to-last tail polynomial (see \eqref{eq:c}), while $h_{\lambda,\mu}^{(N-1)}$ is given in \eqref{eq:h_coeff}. The Schur polynomial at the identity $s_{\mu+(m_{N-1}^{D_{N-1}})}(1^{D_{N-1}})$ evaluates to an integer (see \eqref{eq:schur_identity}). We thus see that $f(P_{D,d})$ has a fully combinatorial form, depending only on the signature $(D,d)$. This completes the proof of Theorem \ref{thm:degree_formula}.

\section{Algorithmic implementation}\label{sec:implementation}
In this section we illustrate the use of the developed theory. We implemented our algorithm in the \texttt{julia} package \texttt{TTVarietyDegree.jl}. We first outline the general use of the package and we mention the used function calls along the derivation. Furthermore, we illustrate the degree computation for a specific example with $N=4$. Finally, we conclude the section with a comparison to independent numerical and symbolic methods.\\

Given a TT signature $(D,d)$, we first compute the Schur expansion of the last tail polynomial $F_{N-1}(CC^*)$. This is achieved by calling \texttt{compute\_tail\_polynomial(D,d)}, which executes the recursion outlined in Section \ref{sec:integration}.
The recursion starts with $F_0=1$. For $r=1,\dots,N-1$, we compute~\eqref{eq:F_recursion}
\begin{align*}
    F_r(\Sigma)
&=
\det(\Sigma)^{m_r}\int_{\St(D_r, n_r)}F_{r-1}(L_{M^r}(\Sigma))d\mu(M^r) \\
&= \det(\Sigma)^{m_r}\int_{\St(D_r, n_r)}\sum_{\substack{\lambda\vdash a_{r-1}\\ \ell(\lambda)\le D_{r-1}}}
c_\lambda^{(r-1)}s_\lambda(L_{M^r}(\Sigma))d\mu(M^r) \\
&=\det(\Sigma)^{m_r}\sum_{\substack{\lambda\vdash a_{r-1}\\ \ell(\lambda)\le D_{r-1}}}
c_\lambda^{(r-1)}H^{(r)}_\lambda(\Sigma),
\end{align*}
where the second equality follows from the Schur expansion of $F_{r-1}$ from the previous recursion step and the third from \eqref{eq:H_function}.
The function $H^{(r)}_\lambda(\Sigma)$ is computed via \texttt{compute\_H\_schur(lambda, D\_prev, D\_cur, d\_r)}, which evaluates \eqref{eq:h_coeff} and returns the Schur expansion 
\[
H_\lambda^{(r)}(\Sigma)
=
\sum_\mu h_{\lambda,\mu}^{(r)}s_\mu(\Sigma).
\]
After this, the function \texttt{determinant\_shift} multiplies by $\det(\Sigma)^{m_r}$ (see Lemma \ref{lemma:det_shift}).
\\
The final tail polynomial $F_{N-1}(CC^*)$ is obtained in the last recursion step and using \eqref{eq:super_long_f} we can compute the degree as in \eqref{eq:degree_formula_f_polynomial}.

\subsection{Example: $N=4$} We now turn to the computation of a concrete example. For illustration, we use the signature
\[
D=(D_0,D_1,D_2,D_3,D_4)=(1,2,2,2,1),
\qquad
d=(d_1,d_2,d_3,d_4)=(3,2,2,2).
\]
To keep track of relevant integers we stack them into the following table
\[
\begin{array}{c||c|c|c}
r & m_r & n_r & a_r\\
\hline
1 & 1 & 3 & 2\\
2 & 2 & 4 & 6\\
3 & 2 & 4 & 10
\end{array} 
\]
The basic usage of the \texttt{julia} package \texttt{TTVarietyDegree.jl} is
\[
\begin{array}{l}
\texttt{using TTVarietyDegree} \\
\texttt{D = [1,2,2,2,1]} \\
\texttt{d = [3,2,2,2]} \\
\texttt{P = compute\_tail\_polynomial(D,d)} \\
\texttt{fP = compute\_fP(P, D[end-1], d[end])} \\
\texttt{deg = degree\_TT\_variety(D,d; reduce=false)}.
\end{array}
\]
Here \texttt{reduce} checks if any rank restriction is trivially satisfied, e.g, whenever $D_1 = d_1$ or $D_{N-1} = d_N$, and simplifies the signature accordingly. Here the reduction would result in a shorter signature $D=[1, 2, 2,1]$ and $d=(3, 2, 4)$.
The projective dimension is computed by
\begin{equation*}
    \dim_{\C} V_{D,d}
=
\sum_{r=1}^{3} D_rn_r + (D_3d_4 
-1) - \sum_{r=1}^{3}D_r^2 = 13 \,.
\end{equation*}
The algorithm executes the recursion for the Schur expansion of $F_3$:
\paragraph{Step $r=1$}
We get $F_1(\Sigma)
=
\det(\Sigma)
=
s_{(1,1)}(\Sigma)$.

\paragraph{Step $r=2$}
Since $F_1(\Sigma)=s_{(1,1)}(\Sigma)$, the routine calls \texttt{compute\_H\_schur($\lambda=(1,1), D_1=2, D_2=2, d_2=2$)}. This yields
\begin{equation*}
    H_{(1,1)}^{(2)}(\Sigma)
=
\frac{1}{10}s_{(2)}(\Sigma)
+
\frac{1}{2}s_{(1,1)}(\Sigma).
\end{equation*}
Since \(m_2=2\), there is a determinantal shift by \((2,2)\) and
\[
F_2(\Sigma)
=
\frac{1}{10}s_{(4,2)}(\Sigma)
+
\frac{1}{2}s_{(3,3)}(\Sigma).
\]

\paragraph{Step $r=3$}
Now \(F_2\) has two Schur terms, so the code computes $H_{(4,2)}^{(3)}$ and $H_{(3,3)}^{(3)}$. The output of  \texttt{compute\_H\_schur($\lambda=(4,2), D_2=2, D_3=2, d_3=2$)} is
\[
H_{(4,2)}^{(3)}
=
\frac{1}{84}s_{(6)}
+
\frac{3}{28}s_{(5,1)}
+
\frac{3}{7}s_{(4,2)}
+
\frac{1}{5}s_{(3,3)},
\]
and \texttt{compute\_H\_schur($\lambda=(3,3), D_2=2, D_3=2, d_3=2$)} yields
\[
H_{(3,3)}^{(3)}
=
\frac{1}{420}s_{(6)}
+
\frac{3}{140}s_{(5,1)}
+
\frac{3}{35}s_{(4,2)}
+
\frac{1}{5}s_{(3,3)}.
\]
Therefore, averaging over $\St(D_3, n_3)$ gives
\begin{align*}
    \int_{\St(D_3, n_3)}\left(\frac{1}{10}s_{(4,2)}(L_{M^3}(\Sigma))
+
\frac{1}{2}s_{(3,3)}(L_{M^3}(\Sigma))\right)d\mu(M^3)
&=
\frac{1}{10}H_{(4,2)}^{(3)}
+
\frac{1}{2}H_{(3,3)}^{(3)} \\
&=
\frac{1}{420}s_{(6)}
+
\frac{3}{140}s_{(5,1)}
+
\frac{3}{35}s_{(4,2)}
+
\frac{3}{25}s_{(3,3)}.
\end{align*}
Finally, \(m_3=2\), so the determinant shifts  by \((2,2)\) and we get
\begin{align*}
    P_{D,d}(C)
=
F_3(CC^*) =
\frac{1}{420}s_{(8,2)}(CC^*)
+
\frac{3}{140}s_{(7,3)}(CC^*)
+
\frac{3}{35}s_{(6,4)}(CC^*)
+
\frac{3}{25}s_{(5,5)}(CC^*).
\end{align*}
Given the Schur expansion of the final tail factor, we compute $f(P_{D,d})$ by calling \texttt{compute\_fP($P_{D,d}, D_3, d_4$)}. Following Proposition \ref{prop:f_schur}, this is a combinatorial term evaluating to $f(P_{D,d})=79488$.
\\
Using $\deg(\operatorname{Gr}(2,3))=1$ and $\deg(\operatorname{Gr}(2,4))=2$ together with \eqref{eq:degree_formula_f_polynomial} we get
\begin{align*}
    \deg(V_{D,d}) = \frac{
\deg(\operatorname{Gr}(2,3))
\deg(\operatorname{Gr}(2,4))
\deg(\operatorname{Gr}(2,4))
}{
(2\cdot 1)!(2\cdot 2)!(2\cdot 2)!
}f(P_{D,d}) 
=
\frac{1\cdot 2\cdot 2}{2!\,4!\,4!}\cdot79488 = 276.
\end{align*}

\subsection{Comparison to independent numerical and symbolic methods}

For consistency, we compare the degrees obtained from \texttt{TTVarietyDegree.jl}
with degrees computed by two independent methods. 
\\

\noindent The first method is a
numerical monodromy computation implemented in
\texttt{HomotopyContinuation.jl} \cite{Breiding_2018}. The second method is
a symbolic computation in \texttt{Macaulay2} \cite{M2}, using standard
commutative algebra methods for computing degrees. \\

The table below contains the degrees for several signatures in the case \(N=4\). In all listed
cases, the degrees computed by our implementation
\texttt{TTVarietyDegree.jl} agree with the degrees obtained by the two
independent methods. These computations provide an external consistency
check for the formula in examples where the alternative methods are still
computationally feasible. For larger values of \(N\), the numerical and
symbolic computations become prohibitively expensive with reasonable
computational resources, while the recursive formula remains applicable (see also Table \ref{tab:degree_ex}).

\begin{table}[H]
\centering
\tiny
\resizebox{\textwidth}{!}{%
\begin{tabular}{c|ccccccc}
\toprule
\multicolumn{1}{c}{} & \multicolumn{7}{c}{\textbf{$d$}} \\
\cmidrule(rl){2-8}
\textbf{$D$}
& \((2,2,3,3)\)
& \((2,3,3,2)\)
& \((2,3,3,3)\)
& \((3,2,2,3)\)
& \((3,2,3,3)\)
& \((3,3,2,3)\)
& \((3,3,3,3)\) \\
\midrule
\((1,1,1,1,1)\)
& 180 & 180 & 630 & 180 & 630 & 630 & 2520 \\

\((1,1,1,2,1)\)
& 216 & 168 & 1080 & 168 & 1080 & 756 & 5940 \\

\((1,1,2,1,1)\)
& 168 & 216 & 1080 & 210 & 756 & 756 & 5940 \\

\((1,1,2,2,1)\)
& 210 & 210 & 4896 & 330 & 1575 & 4446 & 41616 \\

\((1,2,1,1,1)\)
& 210 & 168 & 756 & 168 & 756 & 1080 & 5940 \\

\((1,2,1,2,1)\)
& 360 & 252 & 2376 & 252 & 2376 & 2376 & 30888 \\

\((1,2,2,1,1)\)
& 330 & 210 & 1575 & 330 & 4446 & 1575 & 41616 \\

\((1,2,2,2,1)\)
& 2796 & 1764 & 73332 & 5166 & 87426 & 87426 & 3676860 \\
\bottomrule
\end{tabular}%
}
\caption{Degrees of \(V_{D,d}\) for several signatures with \(N=4\).
The displayed values were computed using \texttt{TTVarietyDegree.jl} and
independently checked using \texttt{HomotopyContinuation.jl} and
\texttt{Macaulay2}.}
\label{tab:degree_comparison_selected}
\end{table}

\printbibliography

\end{document}